\documentclass[12pt]{amsart}

\usepackage{amscd}
\usepackage{amsmath}
\usepackage{amssymb}
\usepackage{amsthm}
\usepackage{epsf}
\usepackage{latexsym}
\usepackage{verbatim}
\usepackage[margin=1.0in]{geometry}
\usepackage[all, cmtip]{xy}
\usepackage{tikz}

\newtheorem{theorem}{Theorem}[section]

\newtheorem{definition}[theorem]{Definition}
\newtheorem{lemma}[theorem]{Lemma}

\newtheorem{corollary}[theorem]{Corollary}
\newtheorem{proposition}[theorem]{Proposition}

\newtheorem{varexample}[theorem]{Example}
\theoremstyle{definition}
\newtheorem{remark}[theorem]{Remark}

\newcommand{\Spec}{\mathrm{Spec}\,}
\newcommand{\PP}{\mathbb{P} }

\newcommand{\OO}{\mathcal{O} }
\newcommand{\QQ}{\mathbb{Q} }
\newcommand{\CC}{\mathbb{C} }

\newcommand{\gquot}{/\!\!/}
\newcommand{\SL}{\operatorname{SL}}

\newcommand{\Chow}{\operatorname{Chow}}

\begin{document}
\title{GIT compactifications of $M_{0,n}$ and flips}
\author{Noah Giansiracusa}
\author{David Jensen}
\author{Han-Bom Moon}
\date{}
\bibliographystyle{alpha}

\maketitle

\begin{abstract}
We use geometric invariant theory (GIT) to construct a large class of compactifications of the moduli space $M_{0,n}$.  These compactifications include many previously known examples, as well as many new ones.  As a consequence of our GIT approach, we exhibit explicit flips and divisorial contractions between these spaces.
\end{abstract}

\section{Introduction}

The moduli spaces of curves ${M}_{g,n}$ and their Deligne-Mumford compactifications $\overline{{M}}_{g,n}$ are among the most ubiquitous and important objects in algebraic geometry.  However, many questions about them remain wide open, including ones that Mumford asked several decades ago concerning various cones of divisors \cite{Mum76,Har87}.  While exploring this topic for $\overline{M}_{0,n}$, Hu and Keel showed that for a sufficiently nice space---a so-called Mori dream space---understanding these cones and their role in birational geometry is intimately related to variations of geometric invariant theory (GIT) quotients \cite{Thaddeus,DH,HK}.  Although it remains unsettled whether $\overline{M}_{0,n}$ is a Mori dream space for $n \ge 7$, the underlying philosophy is applicable nonetheless.

In this paper we explore the birational geometry of $\overline{M}_{0,n}$ and illustrate that VGIT plays a significant role.  We construct a family of modular compactifications of $M_{0,n}$ obtained as  GIT quotients parameterizing $n$-pointed rational normal curves and their degenerations in a projective space.  These compactifications include $\overline{M}_{0,n}$, all the Hassett spaces $\overline{M}_{0,\vec{c}}$, all the previously constructed GIT models, and many new compactifications.

\subsection{The setup}

The Chow variety of degree $d$ curves in $\PP^d$ has an irreducible component parameterizing rational normal curves and their limit cycles.  Denote this by $\Chow(1,d, \PP^d )$ and consider the locus
\[U_{d,n} : = \{ (X, p_1 , \ldots , p_n ) \in \Chow(1,d, \PP^d ) \times ( \PP^d )^n~\vert~p_i \in X \;\forall i \}.\]   There is a natural action of $\SL(d+1)$ on $U_{d,n}$, and the main objects of study in this paper are the GIT quotients $U_{d,n}\gquot\SL(d+1)$ for $n \ge 3$.  These depend on a linearization $L\in\mathbb{Q}_{>0}^{n+1}$ which can be thought of as assigning a rational weight to the curve and each of its marked points.

A preliminary stability analysis reveals that every singular semistable curve is a union of rational normal curves of smaller degree meeting at singularities that are locally a union of coordinate axes (Corollary \ref{cor-geomofstablecurves}).  By considering a certain class of one-parameter subgroups, we derive bounds on the weight of marked points allowed to lie at these singularities and in various linear subspaces (see \S\ref{section:destab}).  Moreover, we show in Proposition \ref{StablePoints} that a rational normal curve with distinct marked points is stable for an appropriate range of linearizations, so there is a convex cone with cross-section $\Delta^\circ \subset \mathbb{Q}^{n+1}$ parameterizing GIT quotients that are compactifications of $M_{0,n}$ (cf. \S\ref{section:EffLin}).  These are related to the Deligne-Mumford-Knudsen compactification as follows:

\begin{theorem}\label{thm:DM-GIT}
Let $d \ge 1$ and $L\in \Delta^{\circ}$. Then:
\begin{enumerate}
	\item The GIT quotient $U_{d,n}\gquot_{L}SL(d+1)$ is
	a compactification of $M_{0,n}$.
	\item There is a regular birational morphism
	\[
		\phi: \overline{M}_{0,n} \rightarrow U_{d,n}\gquot_L\SL(d+1)
	\]
	which preserves $M_{0,n}$.
\end{enumerate}
\end{theorem}

Our technique for proving this is to take an appropriate $\SL(d+1)$-quotient of the Kontsevich space $\overline{M}_{0,n}(\PP^d,d)$ so that every DM-stable curve maps, in a functorial manner, to a GIT-stable curve in $\PP^d$.

\subsection{Chambers, walls, and flips}\label{introsection:maps}

For each fixed $d$, the space of linearizations $\Delta^\circ$ admits a finite wall and chamber decomposition by the general results of VGIT \cite{DH,Thaddeus}.  This endows the birational models we obtain with a rich set of interrelations.  For instance, the quotients corresponding to open chambers map to the quotients corresponding to adjacent walls, and whenever a wall is crossed there is an induced rational map which is frequently a flip.  We undertake a careful analysis of this framework in the context of $U_{d,n}$ and provide a modular description of the maps that arise.

There are two types of walls in the closure of $\Delta^\circ$: interior walls corresponding to changes in stability conditions between open chambers, and exterior walls corresponding to semi-ample linearizations or linearizations with empty stable locus.

Our main results concerning the VGIT of $U_{d,n}$ are the following:
\begin{itemize}
\item \emph{we list all of the GIT walls;}
\item \emph{we classify the strictly semistable curves corresponding to a wall between two chambers and determine the ones with closed orbit;}
\item \emph{we provide necessary and sufficient conditions for the map induced by crossing an interior wall to be i) a divisorial contraction, ii) a flip, or iii) to contract a curve;}
\item \emph{we describe the morphism corresponding to each exterior wall.}
\end{itemize}
Precise statements are provided in \S\ref{section:Smyth} and \S\ref{section:maps}.  The flips we obtain between various models of $\overline{M}_{0,n}$ are quite novel; in fact, it appears that no flips between moduli spaces of pointed genus zero curves have appeared previously in the literature\footnote{That is, a flip in the Mori-theoretic sense of a relatively anti-ample divisor becoming relatively ample; see \cite[Theorem 7.7]{AGS10} for an example of a generalized flip between compactifications.}.  We hope that these can be used to illuminate some previously unexplored Mori-theoretic aspects of the birational geometry of $\overline{M}_{0,n}$.  In particular, we note that the existence of a modular interpretation of these flips, and of the other VGIT maps, is reminiscent of the Hassett-Keel program which aims to construct log canonical models of $\overline{M}_g$ through a sequence of modular flips and contractions.

\subsection{Hassett's weighted spaces}

To illustrate the significance of our unified GIT construction of birational models, consider the Hassett moduli spaces $\overline{M}_{0,\vec{c}}$ of weighted pointed rational curves \cite{Has03}.  For a weight vector $\vec{c}=(c_1,\ldots,c_n)\in\mathbb{Q}_{>0}^n$ with $\sum c_i > 2$, this space parameterizes nodal rational curves with smooth marked points that are allowed to collide if their weights add up to at most 1.  Hassett showed that whenever the weights are decreased, e.g. $\vec{c'}=(c_1',\ldots,c_n')$ with $c'_i \le c_i$, there is a corresponding morphism $\overline{M}_{0,\vec{c}} \rightarrow \overline{M}_{0,\vec{c'}}$.  It has since been discovered that these morphisms are all steps in the log minimal model program for $\overline{M}_{0,n}$. Specifically, the third author shows in \cite{Moo11} that each Hassett space $\overline{M}_{0,\vec{c}}$ is the log canonical model of $\overline{M}_{0,n}$ with respect to the sum of tautological classes $\psi_i$ weighted by $\vec{c}$.

If $\overline{M}_{0,n}$ is indeed a Mori dream space, then by the results of \cite{HK} it would be possible to obtain all log canonical models through VGIT.  Although proving this seems a lofty goal, we are able to deduce the following from our present GIT construction:

\begin{theorem}
For each fixed $n\ge 3$, there exists $d\ge 1$ such that every Hassett space $\overline{M}_{0,\vec{c}}$ arises as a quotient $U_{d,n}\gquot\SL(d+1)$.  Consequently, the log minimal model program for $\overline{M}_{0,n}$ with respect to the $\psi$-classes can be performed entirely through VGIT.
\end{theorem}

\subsection{Modular compactifications}

In the absence of strictly semistable points, each birational model $U_{d,n}\gquot\SL(d+1)$ is itself a fine moduli space of pointed rational curves.  Moreover, this modular interpretation extends  that of the interior, $M_{0,n}$.  A formalism for such compactifications, in any genus, has been introduced by Smyth in \cite{Smy09}.  The basic idea is to define a modular compactification to be an open substack of the stack of all smoothable curves that is proper over $\Spec\mathbb{Z}$.  Smyth shows that there are combinatorial gadgets, called extremal assignments, that produce modular compactifications---and that in genus zero, they produce all of them.  This result can be thought of as a powerful step toward understanding the modular aspects of the birational geometry of $\overline{M}_{0,n}$.  What remains is to determine the maps between these modular compactifications, and for this we can apply our GIT machinery.

In Proposition \ref{prop:GITass}, we identify the extremal assignment corresponding to each GIT linearization without strictly semistable points.  Although this does not yield all modular compactifications (cf. \S\ref{section:ModNotGIT}), it does yield an extensive class of them.  For linearizations that admit strictly semistable points, the corresponding stack-theoretic quotients $[U_{d,n}^{ss}/\SL(d+1)]$ typically are non-separated Artin stacks---so they are not modular in the strict sense of Smyth.  However, they are close to being modular in that they are weakly proper stacks (as in \cite{ASW11}) parameterizing certain equivalence classes of pointed rational curves.  One might call these ``weakly modular'' compactifications.

Recasting the results of \S\ref{introsection:maps} in this light, we begin to see an elegant structure emerge: \emph{Every open GIT chamber in $\Delta^\circ$ corresponds to a modular compactification of $M_{0,n}$, whereas the walls correspond to weakly modular compactifications.  The wall-crossing maps yield relations between the various Smyth spaces that arise in our GIT construction.}  In other words, the GIT chamber decomposition determines which modular compactifications should be thought of as ``adjacent'' in the space of all such compactifications.

\subsection{Previous constructions} In the early 90s, Kapranov introduced two constructions of $\overline{M}_{0,n}$ that have since played an important role in many situations.  He showed that $\overline{M}_{0,n}$ is the closure in $\Chow(1,n-2, \PP^{n-2})$ of the locus of rational normal curves passing through $n$ fixed points in general position \cite{Kap}.  There exist linearizations such that $U_{n-2,n}\gquot\SL(n-1)\cong \overline{M}_{0,n}$, so setting $d=n-2$ in our construction yields a similar construction to Kapranov's---except that instead of fixing the points, we let them vary and then quotient by the group of projectivities.  Kapranov also showed that $\overline{M}_{0,n}$ is the inverse limit of the GIT quotients $(\PP^1)^n\gquot\SL(2)$, which are precisely the $d=1$ case of our construction \cite{KapChow}.  So in a sense, our construction is inspired by, and yields a common generalization of, both of Kapranov's constructions.

\begin{remark}
Kapranov showed that for both of his constructions, one could replace the relevant Chow variety with a Hilbert scheme and the construction remains.  Similarly, we could have used a Hilbert scheme to define a variant of the incidence locus $U_{d,n}$.  By Corollary \ref{cor-geomofstablecurves}, however, every GIT-semistable curve in $U_{d,n}$ is reduced, so the Hilbert-Chow morphism restricts to an isomorphism over the semistable locus.  Therefore, using an asymptotic linearization on the Hilbert scheme would yield GIT quotients isomorphic to those we consider here with the Chow variety.
\end{remark}

The GIT quotients $(\PP^1)^n\gquot\SL(2)$ have made numerous appearances in the literature beyond Kapranov's inverse limit result---they are even included in Mumford's book \cite{git} as ``an elementary example'' of GIT.  The papers \cite{Sim08,GS} introduce and investigate the $d=2$ case of the GIT quotients in this paper.  In \cite{NoahCB}, the first author introduces and studies GIT quotients parameterizing the configurations of points in projective space that arise in $U_{d,n}$, for $1\le d\le n-3$.  These can be viewed as a special case of the current quotients obtained by setting the linearization on the Chow factor to be trivial.  In fact, the GIT quotients studied here appear to include as special cases all GIT quotients of pointed rational curves that have previously been studied.

\subsection{Outline}
\begin{itemize}
\item[\S2:] We explain the GIT setup and prove some preliminary results.  Among these is the fact that all GIT quotients $U_{d,n} \gquot \SL(d+1)$ with linearization in $\Delta^\circ$ are compactifications of $M_{0,n}$ (Prop \ref{StablePoints}).
\item[\S3:] We develop the main tool for studying semistability in these quotients, a weight function that controls the degrees of components of GIT-stable curves.  Using this function we explicitly determine the GIT walls and chambers (Prop \ref{GITWalls}).
\item[\S4:] We show that the GIT quotients $U_{d,n} \gquot \SL(d+1)$ always receive a birational morphism from $\overline{M}_{0,n}$.  This map factors through a Hassett space $\overline{M}_{0,\vec{c}}$ for a fixed weight datum $\vec{c}$ determined by the linearization (Prop \ref{prop:Hassett}).
\item [\S5:] We provide a modular description of all the GIT quotients $U_{d,n} \gquot \SL(d+1)$ (Thm \ref{ModularDescription}).
\item[\S6:] We describe the rational maps between these spaces arising from variation of GIT.  We provide conditions for such a map to be a divisorial contraction (Cor \ref{DivisorialContraction}), a flip (Cor \ref{Flips}), or to contract a curve (Prop \ref{Regular}).
\item[\S7:] We construct several explicit examples of moduli spaces that arise from our GIT construction.  We show that every Hassett space $\overline{M}_{0,\vec{c}}$, including $\overline{M}_{0,n}$, can be constructed in this way (Thm \ref{cor:HassettIso}) and demonstrate an example of variation of GIT for $\overline{M}_{0,9}$ (\S7.3).  We further demonstrate an example of a flip between two compactifications of $M_{0,n}$.
\end{itemize}

\subsection*{Acknowledgements}  We thank K. Chung, A. Gibney, and J. Starr for several helpful conversations regarding this work.  We thank B. Hassett for suggesting the investigation of this GIT construction as a continuation of M. Simpson's thesis \cite{Sim08}, and we thank the referee for very thorough and helpful comments on the paper.

\section{GIT Preliminaries}

\subsection{The cone of linearizations}\label{section:linearizations}

We are interested in the natural action of $\SL(d+1)$ on $U_{d,n} \subseteq \Chow(1,d,\PP^d)\times (\PP^d)^n$.  Since $\SL(d+1)$ has no characters, the choice of a linearization is equivalent to the choice of an ample line bundle.  Each projective space $\PP^d$ has the hyperplane class $\OO_{\PP^d}(1)$ as an ample generator of its Picard group.  The Chow variety has a distinguished ample line bundle $\OO_{Chow} (1)$ coming from the embedding in projective space given by Chow forms.   Therefore, by taking external tensor products we obtain an $\mathbb{N}^{n+1}$ of ample line bundles on $\Chow(1,d,\PP^d)\times (\PP^d)^n$, which we then restrict to $U_{d,n}$.

It is convenient to use fractional linearizations by tensoring with $\mathbb{Q}$.  Moreover,
since stability is unaffected when a linearized line bundle is replaced by a tensor power, we can work with a transverse cross-section of the cone of linearizations:
\[\Delta := \{ ( \gamma , c_1 , c_2 , \ldots, c_n ) \in \mathbb{Q}_{\ge 0}^{n+1} ~|~ (d-1) \gamma + \sum_{i=1}^n c_i = d+1 \}\]  As we will see (Corollary \ref{SmoothWeight}), this ensures all $c_i \le 1$ whenever the semistable locus is nonempty.  This allows us to relate our construction to previous GIT constructions as well as Hassett's spaces, where the point weights are similarly bounded by 1.  We will later restrict to the case that $\gamma < 1$ and $c_i < 1$ for all $i$.  Note that this forces $n \geq 3$.

\subsection{The Hilbert-Mumford numerical criterion}

Let $\lambda : \mathbb{C}^* \to \SL(d+1)$ be a one-parameter subgroup. As in \cite[2.8]{Mum76}, observe that $\lambda$ is conjugate to a subgroup of the form $diag( t^{r_i -k} )$, where $r_0 \geq r_1 \geq \cdots \geq r_d = 0$ and $k = \frac{\sum r_i}{d+1}$.  Choose new coordinates $x_i$ on $\PP^d$ for which $\lambda$ takes this form.  Given a variety $X\subseteq \PP^d$, let $R$ be its homogeneous coordinate ring and $I\subseteq R[t]$ the ideal generated by $\{ t^{r_i}x_i \}_{0 \leq i \leq d}$.  Following \cite[Lemma 1.3]{Schubert}, we denote by $e_{\lambda} (X)$ the normalized leading coefficient of dim$(R[t]/I^m )_m$, where $R[t] = \oplus_{i=1}^{\infty} R_i [t]$ is the grading on $R[t]$ and the normalized leading coefficient of a polynomial $\sum_{i=0}^N a_i x^i$ is $N! a_N$.

The following result is a crucial first step toward the GIT stability analysis conducted subsequently:

\begin{proposition}
\label{NumericalCriterion}
A pointed curve $(X, p_1 , \ldots , p_n ) \in U_{d,n}$ is semistable with respect to the linearization $(\gamma,c_1,\ldots,c_n)\in \Delta$ if, and only if, for every non-trivial 1-PS $\lambda$ with weights $r_i$ as above,
$$ \gamma e_{\lambda} (X) + \sum c_i e_{\lambda} ( p_i ) \leq ( 1 + \gamma ) \sum r_i .$$
It is stable if and only if these inequalities are strict.
\end{proposition}

\begin{proof}
A pointed curve $(X, p_1 , \ldots , p_n )$ is stable (resp. semistable) if and only if, for every 1-PS $\lambda$, the Hilbert-Mumford index $\mu_{\lambda} (X, p_1 , \ldots , p_n )$ is negative (resp. nonpositive). By \cite[Theorem 2.9]{Mum76} and its proof, we see that for the linearization $( \gamma , \vec{0} )$ we have
$$ \mu_{\lambda} (X) = \gamma (e_{\lambda} (X) - \frac{2d}{d+1} \sum r_i ).$$
Similarly, for the linearization $( 0, \vec{c} )$, we have
$$ \mu_{\lambda} ( p_1 , \ldots , p_n ) = \sum c_i e_{\lambda} ( p_i ) - \frac{\sum c_i}{d+1} \sum r_i .$$
By the linearity of the Hilbert-Mumford index, we therefore have
$$ \mu_{\lambda} (X, p_1 , \ldots , p_n ) = \gamma e_{\lambda} (X) + \sum c_i e_{\lambda} ( p_i ) - ( \frac{2d}{d+1} \gamma + \frac{\sum c_i}{d+1} ) \sum r_i $$
$$ = \gamma e_{\lambda} (X) + \sum c_i e_{\lambda} ( p_i ) - ( 1 + \gamma ) \sum r_i, $$
where the last equality follows from the assumption that the linearization vector lies in the cross-section $\Delta$ (cf. \S\ref{section:linearizations}).
\end{proof}

\subsection{Destabilizing one-parameter subgroups}\label{section:destab}

There is one particularly simple type of 1-PS that is sufficient for most of our results.

\begin{proposition}\label{Prop:Flag1PS}
Consider the $k$-dimensional linear subspace $V := V(x_{k+1},x_{k+2},\ldots,x_d) \subset \PP^d$, and let $\lambda_V$ be the 1-PS with weight vector $(1,1, \ldots 1, 0, \ldots , 0)$, where the first $k+1$ weights are all one.  For $X \in \Chow(1,d,\PP^d)$, write $X = X(V) \cup Y$, where $X(V)$ is the union of irreducible components of $X$ contained in $V$.  Then $X$ is semistable with respect to $\lambda_V$ if and only if
$$ \gamma (2 \deg X(V) + e_{\lambda} (Y)) + \sum_{p_i \in V} c_i \leq (k+1)(1 + \gamma ).$$
\end{proposition}

\begin{proof}
This follows from Proposition \ref{NumericalCriterion} and \cite[Lemma 1.2]{Schubert}.
\end{proof}

In most cases we will take $V$ to be a subspace containing some component of $X$, with each of the other irreducible components of $X$ meeting this subspace transversally.  In this case, $e_{\lambda} (Y) = \sum_{Z \subset Y} \vert Z \cap V \vert$, where the sum is over the irreducible components of $Y$.

We first consider the extreme cases $k=d-1$ and $k=0$.  The former leads to instability of degenerate curves, whereas the latter leads to upper bounds on the weight of marked points at smooth and singular points of semistable curves.

\begin{proposition}
\label{prop-degcurveunstable}
A pointed curve $(X,p_1,\ldots,p_n)\in U_{d,n}$ is unstable if $X$ is contained in a hyperplane $\PP^{d-1} \subset \PP^{d}$.
\end{proposition}

\begin{proof}
We may assume that $\PP^{d-1} = V(x_{d})$.  Consider the 1-PS in Proposition \ref{Prop:Flag1PS} with $V := \PP^{d-1}$.  Clearly $X(V)=X$, $Y=\emptyset$, and $\sum_{p_i\in V}c_i = \sum_{i=1}^n c_i = d+1 - (d-1)\gamma$, so
\begin{eqnarray*}
	\gamma(2\deg X(V)+ e_{\lambda} (Y)) + \sum_{p_{i}\in V}c_{i}
	&=&	2d\gamma + (d+1) - (d-1)\gamma\\
	&=& (d+1)(1+\gamma) > 	d(1+\gamma),
\end{eqnarray*}
hence $\lambda_V$ destabilizes $(X,p_1,\ldots,p_n)$.
\end{proof}

Consequently, GIT-semistable curves are geometrically quite nice:

\begin{corollary}
\label{cor-geomofstablecurves}
A semistable pointed curve
$(X, p_{1}, \cdots, p_{n})$ has the following properties:
\begin{enumerate}
	\item Each irreducible component is a rational normal curve in the projective space that it spans.
	\item The singularities are at worst multinodal (analytically locally the union of
	coordinate axes in $\mathbb{C}^{k}$).
	\item Every connected subcurve of degree $e$
	spans a $\PP^{e}$.
\end{enumerate}
\end{corollary}

\begin{proof}
It is proved in \cite[Lemma 13.1]{Art76} that these properties hold for all non-degenerate curves of degree $d$ in $\PP^d$.
\end{proof}

By setting $k=0$ in Proposition \ref{Prop:Flag1PS}, we obtain the following:

\begin{proposition}
\label{SingularWeight}
The total weight of the marked points at a singularity of multiplicity $m$ on a GIT-stable curve cannot exceed $1-(m-1) \gamma$.
\end{proposition}

\begin{proof}
Suppose the singularity occurs at the point $p = (1,0, \ldots , 0)$ and set $k=0$.  Then $X(p) = \emptyset$ and $e_{\lambda} (Y) = \mu_p X = m$.  If $X$ is stable, then by Proposition \ref{Prop:Flag1PS} we have
$$ \gamma m + \sum_{p_i = p} c_i < 1 + \gamma,$$
from which the result follows.
\end{proof}

\begin{corollary}
\label{SmoothWeight}
The total weight of the marked points at a smooth point, or indeed at any point, of a GIT-stable curve cannot exceed 1.
\end{corollary}

\begin{corollary}
\label{Singularities}
A GIT-stable curve cannot have a singularity of multiplicity $m$ unless $\gamma < \frac{1}{m-1}$.
\end{corollary}

\begin{proof}
This follows from the fact that the minimum total weight at a point is zero.
\end{proof}

\begin{corollary}\label{cor:smoothcurves}
If $\gamma \ge 1$, then every GIT-stable curve is smooth.
\end{corollary}

It would be nice at this point to have a result saying that a pointed curve $(X, p_1 , \ldots , p_n ) \in U_{d,n}$ is semistable if and only if, for all subcurves $Y \subset X$, the degree of $Y$ satisfies some formula involving $\gamma$, the weights of the marked points on $Y$, and the number of intersection points $\vert Y \cap \overline{X \smallsetminus Y} \vert$.  As we will see in Proposition \ref{DegreeOfTails}, such a formula exists in the case that $Y$ is a tail of $X$ -- that is, when $\vert Y \cap \overline{X \smallsetminus Y} \vert = 1$.  When $\vert Y \cap \overline{X \smallsetminus Y} \vert > 1$, however, the degree of $Y$ also depends on the distribution of marked points amongst the connected components of $\overline{X \smallsetminus Y}$, as will be shown in Proposition \ref{DegreeOfComponents}.  This is enough to describe a satisfactory stability condition, as we do in Proposition \ref{prop:stabilitydescription}.

\subsection{Existence of a stable point}\label{section:StableExists}
To ensure that GIT quotients of $U_{d,n}$ are compactifications of $M_{0,n}$, it suffices to prove that rational normal curves with configurations of distinct points are stable.  We prove this in several steps.  By Corollary \ref{cor:smoothcurves}, the quotients with $\gamma \ge 1$ are rather uninteresting, so we assume henceforth that $\gamma < 1$.  We begin with the simple case where all of the weights $c_i$ are relatively small.

\begin{lemma}
\label{EasyCase}
Let $(\gamma , \vec{c}) \in \Delta$ satisfy $\gamma <1$ and $0 < c_i < 1 - \gamma $ $\forall i$.  Then every non-degenerate smooth rational curve with distinct marked points is stable.
\end{lemma}

\begin{proof}
Let $X \subset \PP^d$ be a rational normal curve and $p_1 , \ldots , p_n$ distinct points of $X$.  Since all rational normal curves in $\PP^d$ are projectively equivalent, it suffices to show that $(X, p_1 , \ldots , p_n ) \in U_{d,n}$ is stable for the given linearization.  We will show that $(X, p_1 , \ldots , p_n )$ is stable with respect to the linearization $(0, \vec{c})$ and semistable with respect to the linearization $( \gamma , \vec{0} )$.  It then follows from the Hilbert-Mumford numerical criterion that $(X, p_1 , \ldots , p_n )$ is stable with respect to the linearization $( \gamma , \vec{c} )$.

A rational normal curve has reduced degree 1, which is the minimum possible amongst all non-degenerate curves \cite[Theorem 2.15]{Mum76}.  It follows that $X$ is linearly semistable, hence by \cite[Theorem 4.12]{Mum76} it is semistable with respect to the linearization $( \gamma , \vec{0} )$.  Now, let $V \subset \PP^d$ be a $k$-dimensional linear space.  Since any collection of $n$ distinct points on a rational normal curve are in general linear position, we see that
$$ \sum_{p_i \in V} c_i \leq \sum_{p_i \in V} (1- \gamma ) \leq (k+1)(1- \gamma ) < (k+1) \frac{\sum_{i=1}^n c_i}{d+1} .$$
Hence $(p_1 , \ldots , p_n )$ is stable for the linearization $(0, \vec{c})$, by \cite[Example 3.3.24]{DH}.
\end{proof}

We now tackle the more general case.

\begin{proposition}
\label{StablePoints}
Let $(\gamma,\vec{c})\in\Delta$ satisfy $\gamma < 1$ and $0 < c_i < 1$, $i=1,\ldots,n$.  Then every smooth rational curve with distinct marked points is stable, hence $U_{d,n} \gquot_{\gamma, \vec{c}}\SL(d+1)$ compactifies $M_{0,n}$.
\end{proposition}

\begin{proof}
If $c_i < 1 - \gamma $ for all $i$, then the result holds by Lemma \ref{EasyCase} above.  We prove the remaining cases by induction on $d$, the case $d=2$ having been done in \cite{GS}.  Let $(X, p_1 , \ldots , p_n )$ be smooth with distinct points, and assume without loss of generality that $c_1 \geq c_i$ for all $i$ and that $c_1 > 1- \gamma$.  Let $\lambda : \mathbb{C}^* \to \SL(d+1)$ be a 1-PS acting with normalized weights $r_0 \geq r_1 \geq \cdots \geq r_d = 0$ in the sense of \S 2.2, and write $x_i$ for homogeneous coordinates on $\PP^d$ on which $\lambda$ acts diagonally.  We show in Lemma \ref{ReductionStep} below that it is sufficient to consider the situation $p_1 = (1,0,0, \ldots , 0)$, so let us consider this case now.

Let $f_i$ be the restriction of $x_i$ to $X$, which is a homogeneous polynomial of degree $d$ on $X \cong \PP^1$. Write $\pi (X) \subset \PP^{d-1}$ for the image of $X$ under linear projection from $p_1$ and $\lambda^{(d)} : \mathbb{C}^* \to \SL(d)$ for the 1-PS with weights $r_i$, $i>0$, diagonalized with respect to the homogeneous coordinates $x_i$, $i>0$.  By changing homogeneous coordinates $[x,y]$ on $\PP^1$, we assume that $p_1$ is the image of the point $[0:1] \in \PP^1$ under the map $\PP^1 \to \PP^d$ given by the $f_i$'s.  Notice that
$$ e_{\lambda^{(d)}} ( \pi ( p_i )) = \min \{ r_j \vert j >0, f_j ( p_i ) \neq 0 \} \geq \min \{ r_j \vert f_j ( p_i ) \neq 0 \} = e_{\lambda} ( p_i ) $$
$$ e_{\lambda^{(d)}} ( \pi ( p_1 )) = r_a := \min \{ r_j \vert j >0, \frac{f_j}{x} ( p_1 ) \neq 0 \} \leq r_0 .$$

We now show that
$$ e_{\lambda} (X) \leq e_{\lambda^{(d)}} ( \pi (X)) + r_0 + r_a .$$
To see this, note that the polynomials $g_i := \frac{f_i}{x}$ for $i>0$ form a basis for homogeneous polynomials of degree $d-1$.  Let $J$ denote the ideal in $\mathbb{C}[x,y]$ generated by the $f_i$'s for all $i>0$ and $J'$ the ideal in $\mathbb{C}[x,y,t]$ generated by the $t^{r_i} f_i$'s for all $i>0$.  Then $J^m$ consists of all polynomials that vanish at $[0,1]$ to order at least $m$, so dim$\mathbb{C}[x,y]_{md} / J^m = m$.  Since the polynomials $f_0^k f_a^{m-k}, 1 \le k \le m$ each have different order of vanishing at $[0,1]$, they are linearly independent and hence form a basis for this vector space.  Thus, if $I$ is the ideal generated by $t^{r_i} f_i$, we see that the vector space $[ \mathbb{C}[x,y,t]/I^m ]_{md}$, modulo those polynomials that vanish at $[0,1]$ to order at least $m$, is spanned by the linearly independent polynomials $t^j f_0^k f_a^{m-k}$ for $j < kr_0 + (m-k)r_a$.  In other words,
$$ \dim ( \mathbb{C}[x,y,t]/I^m )_{md} \leq \dim ( \mathbb{C} [x,y,t]/ (t^{r_0 k + r_a (m-k)} f_0^k f_a^{m-k}, J'^m ))_{md} $$
$$ \leq \sum_{k=1}^m r_0 k + r_a (m-k) + \dim ( \mathbb{C} [x,y,t]/ (t^{r_i} g_i )^m)_{m(d-1)} $$
$$ \leq {{m+1}\choose{2}} r_0 + {{m}\choose{2}} r_a  + \dim ( \mathbb{C} [x,y,t]/ (t^{r_i} g_i )^m)_{m(d-1)} .$$
Taking normalized leading coefficients, we obtain the formula above.

It follows that
$$ \gamma e_{\lambda} (X) + \sum_{i=1}^n c_i e_{\lambda} ( p_i ) \leq \gamma (e_{\lambda^{(d)}} ( \pi (X)) + r_0 + r_a ) + c_1 r_0 + \sum_{i=2}^n c_i e_{\lambda^{(d)}} ( \pi ( p_i )) .$$
By induction, however, we know that
$$ \gamma e_{\lambda^{(d)}} (\pi (X)) + (c_1 - (1- \gamma )) r_a + \sum_{i=2}^n c_i e_{\lambda^{(d)}} ( \pi ( p_i )) < (1 + \gamma ) \sum_{j=1}^d r_j .$$
It follows that the expression above is smaller than
$$ (1 + \gamma ) \sum_{j=1}^d r_j - (c_1 - (1 - \gamma ))r_a + \gamma r_0 + c_1 r_0 + \gamma r_a \leq (1+ \gamma ) \sum_{j=0}^d r_j$$
as desired.  The result then follows from Lemma \ref{ReductionStep} below.
\end{proof}

\begin{lemma}
\label{ReductionStep}
Let $X$ be a smooth rational normal curve, $p_1 , \ldots , p_n \in X$ distinct, $\lambda : \mathbb{C}^* \to \SL(d+1)$ a 1-PS, and $x_i$ coordinates on $\PP^d$ so that $\lambda$ is normalized as in \S 2.2.  Furthermore, assume that $c_1 \geq c_i$ for all $i$ and $c_1 > 1 - \gamma$.  Then there is a smooth rational normal curve $X'$ with $n$ distinct points $p_1 ' , \ldots , p_n '$ on $X'$ and 1-PS $\lambda '$ such that $p_1' = (1,0, \ldots , 0)$ and
$$ \gamma e_{\lambda} (X) + \sum c_i e_{\lambda} ( p_i ) \leq \gamma e_{\lambda '} (X') + \sum c_i e_{\lambda '} ( p_i ' ) .$$
\end{lemma}

\begin{proof}
Let $V_k \subset \PP^d$ be the $k$-dimensional linear space cut out by $x_{k+1} = x_{k+2} = \cdots = x_d = 0$.  We let $k$ be the smallest integer such that $X \cap V_k$ is non-empty, and write $\lambda '$ for the 1-PS acting with weights $(r_k , r_k , \ldots , r_k , r_{k+1} , \ldots , r_d )$.  Note that $\sum_{i=1}^n c_i e_{\lambda} ( p_i ) = \sum_{i=1}^n c_i e_{\lambda '} ( p_i )$.

We claim that $e_{\lambda} (X) = e_{\lambda '} (X)$ as well.  Indeed, let $W$ denote the linear series on $X \cong \PP^1$ generated by $x_k , \ldots , x_d$.  By assumption, $W$ is basepoint-free, so it contains a basepoint-free pencil.  Using the basepoint-free pencil trick, we see that the map
$$ W \otimes H^0 (X, \OO ((m-1)d)) \to H^0 (X, \OO (md)) $$
is surjective for all $m \geq 2$.  By induction on $m$, the map
$$ Sym^{m-1} W \otimes H^0 (X, \OO (d)) \to H^0 (X, \OO (md)) $$
is surjective as well.  It follows that dim$(R[t]/I^m )_m$ depends only linearly on $r_i$ for all $i<k$.  In other words, these $r_i$'s do not contribute to the normalized leading coefficient, so $e_{\lambda} (X) = e_{\lambda '} (X)$.  Moreover, since the first $k+1$ weights are same, by using an element $g$ of $\mathrm{PGL}(d+1)$ which preserves $x_{k+1}, \cdots, x_{d}$, we can take a smooth rational normal curve $X' := g\cdot X$ such that $e_{\lambda'}(X) = e_{\lambda'}(X')$ and $(1,0,0, \ldots , 0) \in X'$.

Next, relabel the points as follows:
\begin{displaymath}
p_i ' = \left\{ \begin{array}{ll}
(1,0,0, \ldots , 0) & \textrm{if $i=1$}\\
p_1 & \textrm{if $p_i = (1,0,0, \ldots , 0)$}\\
p_i & \textrm{otherwise}
\end{array} \right.
\end{displaymath}
Note that $\sum_{i=1}^n c_i e_{\lambda} ( p_i ) \leq \sum_{i=1}^n c_i e_{\lambda} ( p_i ')$.  In particular, if $p_i = (1,0,0, \ldots , 0)$ for some $i \neq 1$, then since $c_1 \geq c_i$ and $r_0 \geq r_j$ for all $j$, we have
$$ r_0 c_1 + r_k c_i = (r_0 - r_k ) c_1 + r_k c_i + r_k c_1 \geq r_0 c_i + r_k c_1 .$$
This concludes the proof.
\end{proof}

Note that if $c_i > 1$ for any $i$, then no element of $U_{d,n}$ is semistable by Corollary \ref{SmoothWeight}.  The only remaining case, therefore, is when $c_i = 1$ for some $i$.  In this case we will see that every semistable point is strictly semistable, and the resulting quotient is a compactification of $M_{0,n}$ if and only if $d$ is larger than the number of $i$'s for which equality holds.  We delay the proof of this until \S\ref{section:BoundaryQuots}.

\subsection{The space of effective linearizations}\label{section:EffLin}
Recall (cf. \S\ref{section:linearizations}) that we have been working with the cross-section $\Delta$ of the cone of linearizations defined by $(d-1) \gamma + \sum_{i=1}^n c_i = d+1$.  As we remarked earlier, the quotients we are interested in satisfy $\gamma <1$, since otherwise all stable curves are isomorphic to $\PP^1$.  Moreover, by
 Corollary \ref{SmoothWeight} we can assume that $c_i \le 1$ for all $i$.  In fact, by Proposition \ref{StablePoints} we know that if $c_i < 1$ for all $i$ then the linearization $(\gamma,\vec{c})$ is effective, i.e., the semistable locus is nonempty.  To avoid boundary issues such as non-ample linearizations, it is convenient to assume also that $c_i > 0$ for all $i$.  Therefore, we are led to the following space of effective linearizations:
 \[\Delta^\circ := \{ ( \gamma , c_1, \ldots, c_n ) \in \mathbb{Q}^{n+1} ~|~ 0 < \gamma < 1 , 0 < c_i < 1, (d-1) \gamma + \sum_{i=1}^n c_i = d+1 \}.\]  \emph{This is the space of linearizations of most interest to us}.  By Proposition \ref{StablePoints}, $U_{d.n}\gquot_L\SL(d+1)$ is a compactification of $M_{0,n}$ for any $L\in\Delta^\circ$.

\section{Degrees of components in stable curves}\label{section:degrees}

In this section we apply the stability results of the previous section to get a fairly explicit description of the pointed curves $(X,p_1,\ldots,p_n)$ corresponding to stable points of $U_{d,n}$.  Specifically, we show that for a generic linearization, GIT stability completely determines the degrees of subcurves of $X$.  This is then used to describe the walls in the GIT chamber decomposition of $\Delta^\circ$.

We begin by defining a numerical function that will be useful for describing the degrees of subcurves.  First, some notation: given a linearization $(\gamma,\vec{c})$ and a subset $I\subset [n]$, we set \[c_I := \sum_{i\in I} c_i \text{ and } c := \sum_{i=1}^n c_i.\]

\subsection{Weight functions}\label{section:WeightFcn}
Consider the function
$$ \varphi : 2^{[n]} \times \Delta^\circ \to \mathbb{Q} \hspace{1 in} \varphi (I, \gamma , \vec{c} ) = \frac{c_I - 1}{1 - \gamma }. $$
For a fixed linearization $(\gamma,\vec{c})\in\Delta^\circ$, we define
\begin{displaymath}
\sigma (I) = \left\{ \begin{array}{ll}
\lceil \varphi (I, \gamma , \vec{c} ) \rceil & \textrm{if $1 \leq c_I \leq c-1$}\\
0 & \textrm{if $c_I < 1$}\\
d & \textrm{if $c_I > c-1$}
\end{array} \right.
\end{displaymath}

Before relating this to the degrees of subcurves in GIT stable curves, let us make a few elementary observations:

\begin{lemma}\label{lem:sigmdef}
For any $I\subset [n]$, we have $\sigma(I)\in\{0,1,\ldots,d\}$.  If $\sigma(I)=d$, then $c_I > c-1$.
\end{lemma}

\begin{proof}
It is enough to show that $\varphi(I,\gamma,\vec{c}) \le d-1$ whenever $1 \le c_I \le c-1$.  But in this case we have \[\varphi(I,\gamma,\vec{c}) = \frac{c_I - 1}{1-\gamma} \le \frac{c-2}{1-\gamma} = \frac{(d+1-(d-1)\gamma) - 2}{1-\gamma} = d - 1,\] so this indeed holds.
\end{proof}

\begin{lemma}\label{lem:subadd}
For any collection of disjoint subsets $I_1,\ldots,I_m\subset [n]$, \[\sigma ( \bigcup_{j=1}^m I_j ) \geq \sum_{j=1}^m \sigma ( I_j ).\]
\end{lemma}

\begin{proof}
The statement is trivial for $m = 1$, so assume $m \ge 2$.  Note that if $\sigma ( I_j ) = 0$ for any $j$, then it does not contribute to the sum, so we may ignore it.  If there is a $j$ with $c_{I_j} > c-1$, then by the disjointness hypothesis we have  $c_{I_k} < 1$, and hence $\sigma(I_k) = 0$, for all $k\ne j$.  Therefore, we are reduced to the case that $\sigma(I_j) = \lceil \varphi(I_j,\gamma,\vec{c})\rceil$ for every $j$.  In this case, since $\frac{1}{1 - \gamma} \geq 1$, we have
\begin{eqnarray*}
	\sum_{j=1}^m \sigma ( I_j )
	&=& \sum_{j=1}^m \lceil \frac{ c_{I_j} - 1}{1 - \gamma} \rceil
	< \sum_{j=1}^{m}\left( \frac{c_{I_{j}}-1}{1-\gamma} + 1\right)\\
	&=& \frac{\sum_{j=1}^{m}c_{I_{j}}-1}{1-\gamma} -
	\frac{m-1}{1-\gamma} +1
	\le \frac{\sum_{j=1}^{m}c_{I_{j}}-1}{1-\gamma}\\
	&\le& \lceil \frac{\sum_{j=1}^{m}c_{I_{j}}-1}{1-\gamma} \rceil
	= \lceil \frac{c_{I_{1} \cup \cdots \cup I_{m}}-1}{1-\gamma} \rceil,
\end{eqnarray*}
which by definition is $\sigma ( \bigcup_{j=1}^m I_j )$.
\end{proof}

Perhaps most significantly, $\sigma$ satisfies a convenient additivity property for most linearizations:

\begin{lemma}
\label{Additivity}
If $\varphi (I , \gamma , \vec{c} ) \notin \mathbb{Z}$ for each nonempty $I\subset [n]$, then \[\sigma (I) + \sigma ( I^c ) = d\] for each $I$.
\end{lemma}

\begin{proof}
If $c_I < 1$ then $c_{I^c} = c - c_I > c -1$, so $\sigma(I)+\sigma(I^c) = 0+d = d$.  The case $c_I > c -1$ is analogous, so without loss of generality assume that $c_I$ and $c_{I^c}$ are between 1 and $c-1$.  Then
\begin{eqnarray*}
\sigma ( I^c ) = \lceil \frac{c_{I^c} - 1}{1 - \gamma} \rceil
&=& \lceil \frac{(d+1)-(d-1)\gamma - c_{I}-1}{1-\gamma}\rceil\\
&=&
\lceil d-1- \frac{c_I -1}{1 - \gamma} \rceil = d - \sigma (I),
\end{eqnarray*}
where the last equality uses the non-integrality assumption.
\end{proof}

\subsection{Degrees of tails}
As we show below, the function $\sigma$ computes the degree of a certain type of subcurve.  For notational convenience, given a marked curve $(X,p_1,\ldots,p_n)$ and a subcurve $Y \subset X$, let us set
$$ \varphi (Y, \gamma , \vec{c} ) = \varphi ( \{ i~|~p_i \in Y \} , \gamma , \vec{c} )$$ and similarly for $\sigma(Y)$.

\begin{definition}
Let $X \in \Chow(1,d,\PP^d)$.  A subcurve $Y \subset X$ is called a \textbf{tail} if it is connected and $\vert Y \cap \overline{X \backslash Y} \vert = 1$.
\end{definition}

We do not require tails to be irreducible.  Moreover, the ``attaching point'' of a tail need not be a node.

\begin{proposition}
\label{DegreeOfTails}
For a fixed $(\gamma,\vec{c})\in\Delta^\circ$, suppose that $\varphi (I , \gamma , \vec{c} ) \notin \mathbb{Z}$ for any nonempty $I \subset [n]$.  If $X$ is a GIT-semistable curve and $E \subset X$ a tail, then $\deg (E) = \sigma (E)$.
\end{proposition}

\begin{proof}
Write $r:=\deg (E)$.  The dimension of the linear span of $E$ is $r$ by Corollary \ref{cor-geomofstablecurves}, so we may assume that $E\subset V := V(x_{r+1}, \ldots,x_d) \subset \PP^d$.  Now
$$ \gamma (2 \deg X(V) + e_{\lambda} (Y)) + \sum_{p_i \in V} c_i \geq \gamma (2r + 1) + \sum_{p_i \in E} c_i,$$
so by Proposition \ref{Prop:Flag1PS} we have
$$ \sum_{p_i \in E} c_i \le (r+1)(1 + \gamma ) - \gamma (2r+1) = r+1 - \gamma r ,$$
or equivalently,
$$ r \ge \frac{( \sum_{p_i \in E} c_i ) -1}{1 - \gamma}. $$
Since $r$ is a positive integer, it follows that $r \geq \sigma (E)$.  Note that if $\sigma (E) = d$, then $r > \frac{c-2}{1- \gamma} = d-1$, so the result still holds in this case.

Now, if $E$ is a tail then so is $\overline{X \backslash E}$, hence
$$ \text{deg} ( \overline{X \backslash E} ) \geq \sigma ( \overline{X \backslash E} ) \geq \sigma ( \{ i \vert p_i \notin E \} ) .$$
Thus, by Lemma \ref{Additivity}, $\text{deg} ( \overline{X \backslash E} ) \geq d - \sigma (E)$.  But we know that $r + \text{deg} ( \overline{X \backslash E} ) = d$, so the inequality $r \le \sigma(E)$ also holds.
\end{proof}

\subsection{Arbitrary subcurves}

Removing an irreducible component from a semistable curve in $\Chow(1,d,\PP^d )$ yields a finite collection of tails.  This holds more generally for any connected subcurve.  We can combine this fact with the above result on tails to deduce the following:

\begin{corollary}
\label{DegreeOfComponents}
Suppose that $\varphi (I , \gamma , \vec{c} ) \notin \mathbb{Z}$ for any $\emptyset \ne I \subset [n]$, and let $E\subseteq X$ be a connected subcurve of $(X,p_1,\ldots,p_n) \in U_{d,n}^{ss}$.  Then
$$ \deg (E) = d - \sum \sigma (Y) $$
where the sum is over all connected components $Y$ of $\overline{X \backslash E}$.
\end{corollary}

\begin{proof}
If $Y$ is a connected component of $\overline{X \backslash E}$, then it is a tail.  It follows from Proposition \ref{DegreeOfTails} that $\deg (Y) = \sigma (Y)$.  Since the total degree of $X$ is $d$, we see that $ \deg (E) = d - \sum \sigma (Y) $.
\end{proof}

We now have enough information to completely describe stability of pointed curves in $U_{d,n}$, though we postpone the proof of the following result until \S\ref{section:Morphism}.

\begin{proposition}\label{prop:stabilitydescription}
Let $L = (\gamma, c_{1}, \cdots, c_{n}) \in \Delta$ be such that $U_{d,n}^{ss}(L) = U_{d,n}^{s}(L)$. A pointed curve $(X, p_{1}, \cdots, p_{n}) \in U_{d,n}$ is stable with respect to $L$ if and only if $X \subset \PP^{d}$ is non-degenerate, for any point $p \in X$ with multiplicity $m$, $\sum_{p_{i}=p}c_{i}<1-(m-1)\gamma$, and for any tail $Y \subset X$, $\deg (Y) = \sigma(Y)$.
\end{proposition}

\subsection{GIT Walls}\label{section:walls}

These results are sufficient to determine the wall-and-chamber decomposition of $\Delta^\circ$.  Specifically, for any integer $k$ with $0 \leq k \leq d-1$, the set $\varphi (I, \cdot )^{-1} (k)$ defines a hyperplane in $\Delta^\circ$.  Note that, by additivity,
$$\varphi (I, \cdot )^{-1} (k) = \varphi (I^c , \cdot )^{-1} (d-1-k),$$
but otherwise these hyperplanes are distinct.

\begin{lemma}\label{lem:existencemarkedpoints}
If $( \gamma , \vec{c} )$ is not contained in any hyperplane of the form $\varphi (I, \cdot )^{-1} (k)$, then:
\begin{enumerate}
	\item An irreducible tail $E$ has at least two distinct marked points on its
	smooth locus $E^{sm}$.
	\item An irreducible component $E$ with $|E \cap \overline{X\backslash
	E}| = 2$ has at least one marked point on $E^{sm}$.
\end{enumerate}
\end{lemma}

\begin{proof}
Let $E \subset X$ be an irreducible tail.  Since $E$ has positive degree, by Proposition \ref{DegreeOfTails} we have $\sigma(E) \ge 1$, so by additivity $\sigma(\overline{X \backslash E}) \le d-1$, and hence by definition we see that $\sum_{p_i\in\overline{X \backslash E}} c_i \le c -1$.  By the non-integrality assumption this inequality must be strict, and consequently $\sum_{p_i \in E^{sm}} c_i > 1.$  On the other hand, by Corollary \ref{SmoothWeight}, the sum of the weights at a smooth point of $E$ cannot exceed 1.  It follows that the marked points on $E$ must be supported at 2 or more points of $E$ other than the singular point.

Similarly, let $E \subset X$ be a bridge---a component such that $\vert E \cap \overline{X \backslash E} \vert = 2$.  Let $Y_1 , Y_2$ denote the connected components of $\overline{X \backslash E}$.  If the smooth part of $E$ contains no marked points, then by Lemma \ref{Additivity} we see that $\sigma ( Y_1 ) + \sigma ( Y_2 ) = d$.  Again, since $E$ has positive degree, by Corollary \ref{DegreeOfComponents} this is impossible.
\end{proof}

\begin{proposition}
\label{NoAutomorphisms}
If $( \gamma , \vec{c} )$ is not contained in any hyperplane of the form $\varphi (I, \cdot )^{-1} (k)$, then every semistable pointed curve has trivial automorphism group.
\end{proposition}

\begin{proof}
By Corollary \ref{cor-geomofstablecurves}, every semistable curve is a union of rational normal curves meeting in multinodal singularities. We claim that an automorphism $f$ of a semistable curve $(X, p_{1}, \cdots, p_{n})$ does not permute its irreducible components nontrivially. Indeed, it is straightforward to see that if there is a nontrivial permutation of irreducible components of $X$, then there are two distinct irreducible tails $E_{1}, E_{2}$ such that $f(E_{1}) = E_{2}$. But by (1) of Lemma \ref{lem:existencemarkedpoints}, they have marked points (say $p_{1}$ and $p_{2}$) on their smooth parts. This is impossible because $f(p_{1}) = p_{1} \in E_{1}$. Thus the automorphism $f$ induces automorphisms of its irreducible components, which are isomorphic to $\PP^{1}$.

It follows that such a curve $(X, p_1 , \ldots , p_n )$ has a non-trivial automorphism if and only if it contains either:
\begin{enumerate}
\item  an irreducible tail $E$ with all marked points of its smooth locus $E^{sm}$ supported on at most one point, or
\item  an irreducible component $E$ with $\vert E \cap \overline{X \backslash E} \vert = 2$ such that $E^{sm}$ contains no marked points.
\end{enumerate}
Both cases are impossible due to Lemma \ref{lem:existencemarkedpoints}.
\end{proof}

\begin{corollary}
\label{NoSemistablePoints}
If $( \gamma , \vec{c} )$ is not contained in any hyperplane of the form $\varphi (I, \cdot )^{-1} (k)$, then the corresponding GIT quotient admits no strictly semistable points.
\end{corollary}

\begin{proof}
If $U_{d,n}^{ss}$ contains strictly semistable points, then some of these points must have positive-dimensional stabilizer.  If $(X, p_1 , \ldots , p_n )$ is such a curve, then since $X$ spans $\PP^d$ by Proposition \ref{prop-degcurveunstable}, such a stabilizer cannot fix $X$ pointwise.  It follows that $(X, p_1 , \ldots , p_n )$ admits a positive-dimensional family of automorphisms, contradicting Proposition \ref{NoAutomorphisms}.
\end{proof}

\begin{proposition}
\label{GITWalls}
The hyperplanes $\varphi (I, \cdot )^{-1} (k)$ are the walls in the GIT chamber decomposition of $\Delta^\circ$.
\end{proposition}

\begin{proof}
By Corollary \ref{NoSemistablePoints}, if a linearization does not lie on any of these hyperplanes, then it admits no strictly semistable points.  Hence the GIT walls must be contained in these hyperplanes.  To see that each hyperplane $\varphi (I, \cdot )^{-1} (k)$ yields a wall in $\Delta^\circ$, we must show that the stable locus changes when each such hyperplane is crossed.  But it is clear from the definition that the function $\sigma$ in \S\ref{section:WeightFcn} changes along these hyperplanes, so by Proposition \ref{DegreeOfTails}, GIT stability changes as well.
\end{proof}

\section{From Deligne-Mumford to GIT}\label{section:Morphism}

In this section we prove item (2) of Theorem \ref{thm:DM-GIT}, i.e., that the GIT quotients $U_{d,n} \gquot \SL(d+1)$ receive a birational morphism from the moduli space of stable curves $\overline{M}_{0,n}$.  The main tool we use is the Kontsevich space of stable maps $\overline{M}_{0,n} ( \PP^d , d)$ \cite{FPnotes}.  The basic idea is as follows.  The product of evaluation maps yields a morphism $\overline{M}_{0,n} ( \PP^d , d) \rightarrow (\PP^d)^n$.  By pushing forward the fundamental cycle of each curve under each stable map, there is also a morphism $\overline{M}_{0,n} ( \PP^d , d) \rightarrow \Chow(1,d,\PP^d)$.  By functoriality, one sees that together these yield a morphism \[\phi : \overline{M}_{0,n} ( \PP^d , d) \rightarrow U_{d,n} \subset \Chow(1,d,\PP^d)\times (\PP^d)^n.\]  This map is clearly $\SL(d+1)$-equivariant.  We prove below that for a general linearization $L$ on $U_{d,n}$, there is a corresponding linearization $L'$ on $\overline{M}_{0,n} ( \PP^d , d)$ such that there is an induced
\begin{enumerate}
\item  morphism $\overline{M}_{0,n} ( \PP^d , d)\gquot_{L'}\SL(d+1) \rightarrow U_{d,n}\gquot_L\SL(d+1)$, and
\item  isomorphism $\overline{M}_{0,n} ( \PP^d , d)\gquot_{L'}\SL(d+1) \cong \overline{M}_{0,n}$.
\end{enumerate}
This is enough to draw the desired conclusion:

\begin{lemma}\label{lemma-morphism}
If (1) and (2) above hold for all $L\in\Delta^\circ$ that do not lie on a GIT wall, then for any $L\in\Delta^\circ$ there is a regular birational morphism $\overline{M}_{0,n} \rightarrow U_{d,n}\gquot_L\SL(d+1)$.
\end{lemma}

\begin{proof}
Given $L\in\Delta^\circ$, we can perturb it slightly to obtain a linearization $L_{\epsilon}$ such that stability and semistability coincide.  By general variation of GIT, there is a birational morphism from the $L_\epsilon$-quotient to the $L$-quotient.  Using (1) and (2) we then have \[\overline{M}_{0,n} \cong \overline{M}_{0,n} ( \PP^d , d) \gquot_{L_{\epsilon}'} SL(d+1) \rightarrow	U_{d,n} \gquot_{L_{\epsilon}} SL(d+1) \rightarrow U_{d,n}\gquot_{L} SL(d+1).\] Birationality of this morphism follows from Proposition \ref{StablePoints}.
\end{proof}

\subsection{Equivariant maps and GIT}

Here we prove a generalized form of the result needed for item (1) above.

\begin{lemma}\label{lem-stabilitybirational}
Let $f : X \to Y$ be a $G$-equivariant birational morphism between two projective varieties.  Suppose $X$ is normal, and let $L$ be a linearization on $Y$. Then there exists a linearization $L'$ on $X$ such that
\[
	f^{-1}(Y^{s}(L)) \subset X^{s}(L')
	\subset X^{ss}(L') \subset f^{-1}(Y^{ss}(L)).
\]
\end{lemma}

\begin{proof}
Take an $f$-ample divisor $M$, the existence of which is guaranteed by \cite[5.3, 5.5]{EGA2}.  Since $X$ is normal, some integral multiple of $M$ is $G$-linearized \cite[Corollary 1.6]{git}, so we may assume that $M$ is $G$-linearized.  Let $L' = f^{*}(L^{m}) \otimes M$ for sufficiently large $m$.  Then $L'$ is ample and the above inclusions hold by
\cite[Theorem 3.11]{Hu96}.
\end{proof}

In particular, if $Y^{s}(L) = Y^{ss}(L)$, then $X^{s}(L') = X^{ss}(L') = f^{-1}(Y^{s}(L))$. 

\begin{corollary}
With the same assumptions as the previous lemma, there is an induced morphism of quotients
\[
	\overline{f} : X \gquot_{L'} G \to Y \gquot_{L} G.
\]
\end{corollary}

\begin{proof}
By Lemma \ref{lem-stabilitybirational}, we have $f(X^{ss}(L')) \subset Y^{ss}(L)$, so there is a morphism $X^{ss}(L') \rightarrow Y\gquot_L G$.  This is $G$-invariant, so it must factor through the categorical quotient of $X^{ss}(L')$ by $G$, which is precisely the GIT quotient $X\gquot_{L'}G$.
\end{proof}

\subsection{Invariant maps and unstable divisors}

In this subsection we address item (2) above.  To begin, recall that there is a forgetting-stabilizing map $\pi: \overline{M}_{0,n} ( \PP^d , d) \rightarrow \overline{M}_{0,n}$.  Since this is $\SL(d+1)$-invariant, the universal property of categorical quotients implies that there is an induced map \[\overline{\pi} : \overline{M}_{0,n} ( \PP^d , d)\gquot_{L'}\SL(d+1) \rightarrow \overline{M}_{0,n}.\] for any linearization $L'$.  The main result here is that if $L\in\Delta^\circ$ does not lie on a GIT wall and $L'$ is as in Lemma \ref{lem-stabilitybirational}, then this induced quotient morphism is in fact an isomorphism.  In what follows, we always consider a linearization $L'$ on $\overline{M}_{0,n} ( \PP^d , d)\gquot_{L'}\SL(d+1)$ that is of this form, so that stability and semistability coincide.  To show that $\overline{\pi}$ is an isomorphism, we show that it has relative Picard number zero.

We first recall some divisor classes on $\overline{M}_{0,n}(\PP^{d}, d)$. For $0 \le i \le d$ and $I \subset [n]$, let $D_{i, I}$ be the closure of the locus of stable maps $(f : (C_{1} \cup C_{2}, p_{1}, \cdots, p_{n}) \to \PP^{d})$ such that
\begin{itemize}
	\item the domain of $f$ has two irreducible components $C_{1}, C_{2}$;
	\item $p_{j} \in C_{1}$ if and only if $j \in I$;
	\item $\deg f_{*}C_{1} = i$ (equivalently, $\deg f_{*}C_{2} = d - i$).
\end{itemize}
It is well known that $D_{i, I}$ is codimension one if it is nonempty. By definition, $D_{i, I} = D_{d-i, I^{c}}$ so whenever we write down $D_{i, I}$, we may assume that $|I| \le \frac{n}{2}$.   Note that $D_{i, I} = \emptyset$ if and only if $i = 0$ and $|I| \le 1$.
Also, let
\[
	D_{deg} = \{f : (C, p_{1}, \cdots, p_{n}) \to \PP^{d}~|~
	\mbox{span of $f(C)$ is not $\PP^{d}$}\},
\]
which is a divisor as well.

First, a couple preliminary results:

\begin{lemma}\label{lem-unstabledivisor}
For $0 \le i \le d$ and $I \subset [n]$, if $1 < |I| \le \frac{n}{2}$, at most one of $D_{i, I}$ for $i = 0, 1, \ldots, d$ can be stable.  If $\vert I \vert \leq 1$, then none of the $D_{i,I}$ are stable.
\end{lemma}

\begin{proof}
By Lemma \ref{lem-stabilitybirational} and the stability assumption, to compute stability of $x \in \overline{M}_{0,n} ( \PP^d , d)$, it suffices to consider the stability of $\phi(x) \in U_{d,n}$.

Choose a general point $(f:(C_{1} \cup C_{2}, p_{1}, \cdots, p_{n})
\to \PP^{d})$ in $D_{i, I}$.
Then $f(C_{1}) \subset \PP^{d}$ is a degree $i$ rational normal curve
and $f(C_{2}) \subset \PP^{d}$ is a degree $d-i$ rational normal curve.
(If $i = 0$, then $f(C_{1})$ is a point.)
By dimension considerations, the linear spans of $f(C_{1})$ and $f(C_{2})$ meet at a unique
point, namely $f(C_{1} \cap C_{2})$.
By Proposition \ref{DegreeOfTails}, $f(C_1 \cup C_2 )$ is stable only if $\deg(f \vert_{C_1})=\sigma (I)$ and $\deg(f \vert_{C_2})=\sigma (I^c )$, so at most one $D_{i, I}, i\in\{0,\ldots,d\}$, is stable.  On the other hand, if $I$ contains at most 1 marked point, then $\sigma (I) = 0$, so $D_{i,I}$ is not stable.
\end{proof}

\begin{lemma}\label{lem-PicQuot}
Let $X$ be a normal projective variety with a linearized $\SL(n)$-action, and suppose that $X^{ss} = X^{s}$.  Then \[\mathrm{Pic}(X\gquot \SL(n))_\QQ \cong  \mathrm{Pic}(X^{ss})_\QQ.\]
\end{lemma}

\begin{proof}
Since $X$ is normal, by \cite[Theorem 7.2]{Dol03} we have a canonical exact sequence
\[
	\mathrm{Pic}^{\SL(n)}(X^{ss})
	\stackrel{\alpha}{\to} \mathrm{Pic}(X^{ss}) \to \mathrm{Pic}(\SL(n))
\]
where $\mathrm{Pic}^{\SL(n)}(X^{ss})$ is the group of $\SL(n)$-linearized line bundles.  Thus $\alpha$ is surjective, since $\mathrm{Pic}(\SL(n)) = 0$. Moreover, since $\mathrm{Hom}(\SL(n), \CC^{*})$ is trivial, by \cite[Proposition 1.4]{git} we see that $\alpha$ is injective.  Thus $\mathrm{Pic}^{\SL(n)}(X^{ss}) \cong \mathrm{Pic}(X^{ss})$.

On the other hand, let $\mathrm{Pic}^{\SL(n)}(X^{ss})^{0}$ be the subgroup of $\SL(n)$-linearized line bundles $L$ such that the stabilizer of a point in a closed orbit acts on $L$ trivially. Since any point over $X^{ss} = X^{s}$ has finite stabilizer, $\mathrm{Pic}^{\SL(n)}(X^{ss})^{0}$ has finite index in $\mathrm{Pic}^{\SL(n)}(X^{ss})$ and $\mathrm{Pic}^{\SL(n)}(X^{ss})^{0}_{\QQ} \cong \mathrm{Pic}^{\SL(n)}(X^{ss})_{\QQ}$.  Finally, by Kempf's descent lemma \cite[Theorem 2.3]{DN89}, $\mathrm{Pic}(X\gquot \SL(n)) \cong \mathrm{Pic}^{\SL(n)}(X^{ss})^{0}$. In summary, we have a sequence of isomorphisms
\[
	\mathrm{Pic}(X^{ss})_{\QQ} \cong
	\mathrm{Pic}^{\SL(n)}(X^{ss})_{\QQ} \cong
	\mathrm{Pic}^{\SL(n)}(X^{ss})^{0}_{\QQ} \cong
	\mathrm{Pic}(X \gquot \SL(n))_{\QQ}.
\]
\end{proof}

We now prove the main result.

\begin{proposition}\label{thm-isomorphism}
The map $\overline{\pi} : \overline{M}_{0,n} ( \PP^d , d) \gquot_{L'} \SL(d+1) \to \overline{M}_{0,n}$ is an isomorphism.
\end{proposition}

\begin{proof}
For $d = 1$, this is exactly \cite[Theorem 3.4]{HK}, since $\overline{M}_{0,n}(\PP^{1},1) \cong \PP^{1}[n]$, the Fulton-MacPherson space of $\PP^{1}$. We prove for $d \ge 2$ cases. 

The space $\overline{M}_{0,n} ( \PP^d , d)$ is a normal variety with finite quotient singularities only \cite[Theorem 2]{FPnotes}. Since $\overline{\pi}$ is a birational morphism between two projective varieties, it is projective. Thus there is a $\overline{\pi}$-ample line bundle $A$. Since $\overline{\pi}$ is a birational morphism between two normal varieties,
\[
	\overline{\pi}^{*}: \mathrm{N}^{1}(\overline{M}_{0,n})_{\QQ} \to
	\mathrm{N}^{1}(\overline{M}_{0,n}(\PP^{d}, d)\gquot_{L'}
	\SL(d+1))_{\QQ}
\]
is injective. If $\overline{\pi}$ is not an isomorphism, then there is a curve $C$ that is contracted by $\overline{\pi}$.  Note that $C \cdot A > 0$.  This implies that $\overline{\pi}^{*}$ is not surjective, so to show that $\overline{\pi}$ is an isomorphism it suffices to show that the Picard numbers of both varieties are the same.

By \cite{Kee92}, the Picard number of $\overline{M}_{0,n}$ is $2^{n-1} - {n \choose 2} - 1$.  By \cite[Theorem 2]{Pan99}), the Picard number of $\overline{M}_{0,n} ( \PP^d , d)$, for $d \ge 2$, is $(d+1)2^{n-1} - {n \choose 2}$.  Therefore, it suffices to show that there are $d \cdot 2^{n-1} + 1$ numerically independent unstable divisors.

Take a partition $I \sqcup I^{c}$ of $[n]$. Among $D_{0, I}, D_{1, I}, \ldots, D_{d, I}$, there are at least $d$ unstable divisors by Lemma \ref{lem-unstabledivisor}.  It follows from \cite[Lemma 1.2.3]{Pan99} that these are all numerically independent.  Since degenerate curves in $U_{d, n}$ are unstable by Proposition \ref{prop-degcurveunstable}, their inverse image $D_{deg}$ is unstable, too.  One checks that this divisor is independent of the preceding divisors either by explicitly constructing a curve in $\overline{M}_{0,n} ( \PP^d , d)$ or by using the formula for $D_{\deg}$ in the $n=0$ case in \cite[Lemma 2.1]{CHS} and pulling back to $\overline{M}_{0,n} ( \PP^d , d)$.

Combining this with Lemma \ref{lem-PicQuot}, and writing $\rho$ for the Picard number, we obtain
\begin{eqnarray*}
	\rho(\overline{M}_{0,n} ( \PP^d , d) \gquot_{L'} \SL(d+1)) &=&
	\rho(\overline{M}_{0,n} ( \PP^d , d)^{s}) \\
	&\le& (d+1)2^{n-1}-{n \choose 2} - d2^{n-1}-1\\
	&=& 2^{n-1}-{n \choose 2} - 1=\rho(\overline{M}_{0,n})
\end{eqnarray*}
The opposite inequality holds due to the existence of the birational morphism $\overline{\pi}$.  This completes the proof.
\end{proof}

From the idea of the proof of Proposition \ref{thm-isomorphism}, we can obtain a proof of the stability result in Proposition \ref{prop:stabilitydescription}.

\begin{proof}[Proof of Proposition \ref{prop:stabilitydescription}]
Suppose that $(X, p_{1}, \cdots, p_{n}) \in U_{d,n}^{ss}(L)$.  Then $X \subset \PP^{d}$ is non-degenerate by Proposition \ref{prop-degcurveunstable}. For any point $p \in X$ of multiplicity $m$, $\sum_{p_{i}=p}c_{i} < 1 - (m-1)\gamma$ by Proposition \ref{SingularWeight}.  Also, for any tail $Y \subset X$, $\deg(Y) = \sigma(Y)$ by Proposition \ref{DegreeOfTails}.

Conversely, let $(X, p_{1}, \cdots, p_{n}) \in U_{d,n}$ be a pointed curve satisfying the assumptions above.  Let $(f : (C, x_{1}, \cdots, x_{n}) \to \PP^{d}) \in \overline{M}_{0,n}(\PP^{d}, d)$ be a stable map such that $\phi(f) = (X, p_{1}, \cdots, p_{n})$ where $\phi : \overline{M}_{0,n}(\PP^{d}, d) \to U_{d,n}$ be the cycle map. For an irreducible component $D \subset C$, if $f(D) \subset X$ is not a point, we claim that $D$ has at least three special points (singular points and marked points). Indeed, if $Y = f(D)$ is a tail, then $\sigma(f(D)) \ge 1$ or equivalently, $\sum_{p_{i} \in f(D)^{sm}}c_{i} > 1$ because $\sum_{p_{i} \in f(D)}c_{i} > 2- \gamma$ and on the unique singular point $p$ of $f(D)$, $\sum_{p_{i} = p}c_{i} < 1 - \gamma$ by Proposition \ref{SingularWeight}. Since the sum of the weights at a smooth point is at most one, there must be at least two marked points on $f(D)^{sm}$.  Similarly, if $f(D)$ is a bridge, $f(D)$ can be regarded as a complement of two (possibly reducible) tails $E_{1}$ and $E_{2}$. If there is no marked points on $f(D)^{sm}$, then
\[
	\sigma(E_{1}) + \sigma(E_{2}) = \sigma(\{p_{i}\in E_{1}\})
	+ \sigma(\{p_{i}\in E_{2}\}) = d
\]
by Lemma \ref{Additivity},  thus $f(D)$ must be a point. It follows that a bridge $f(D)$ must have a marked point on $f(D)^{sm}$.  In the remaining cases, $f(D)$ has at least three singular points.

If $f(D)$ is a point, there exist at least three special points since $f$ is a stable map. Thus the domain $(C, x_{1}, \cdots, x_{n})$ is already an $n$-pointed stable rational curve. So $\pi(f) = (C, x_{1}, \cdots, x_{n})$ for $\pi : \overline{M}_{0,n}(\PP^{d}, d) \to \overline{M}_{0,n}$.

Since $\pi : \overline{M}_{0,n}(\PP^{d}, d)^{ss}(L') \to \overline{M}_{0,n}$ is surjective, there exists
\[
	(\tilde{f} : (C', x_{1}', \cdots, x_{n}') \to \PP^{d}) \in
	\pi^{-1}(C, x_{1}, \cdots, x_{n}) \cap
	\overline{M}_{0,n}(\PP^{d}, d)^{ss}(L').
\]
We claim that $C' \cong C$ and $\tilde{f} \cong f$ up to projective equivalence. If $C' \not\cong C$, then there exists a nontrivial contraction $c : C' \to C$ and a contracted irreducible component $D' \subset C'$ which has at most two special points. Note that for every (possibly reducible) tail $D \subset C'$ we can determine $\deg \tilde{f}(D)$ by Lemma \ref{lem-unstabledivisor} and it must be equal to $\sigma(D) = \sigma(c(D)) = \deg f(c(D))$. In particular, the sum of degrees of $\tilde{f}$ on the non-contracted irreducible components is already $d$ and $\deg \tilde{f}(D') = 0$. This is impossible since $\tilde{f}$ is a stable map so a degree zero component must have at least three special points.
The projective equivalence of $\tilde{f}$ and $f$ can be shown by induction on the number of irreducible components, since for each irreducible component $D \subset C$, $f(D)$ is a rational normal curve in its span and there is a unique rational normal curve up to projective equivalence.

Therefore, $f$ is in the $\SL(d+1)$-orbit of $\tilde{f}$.  Hence $f \in \overline{M}_{0,n}(\PP^{d}, d)^{ss}(L')$.  From $\phi^{-1}(U_{d,n}^{ss}(L)) = \overline{M}_{0,n}(\PP^{d},d)^{ss}(L')$ (Lemma \ref{lem-stabilitybirational}), we have $\phi(f) = (X, p_{1}, \cdots, p_{n}) \in U_{d,n}^{ss}(L)$.
\end{proof}

\begin{remark}\label{rem:nostabilizing}
This proof tells us that if $L \in \Delta^{0}$ is a linearization admitting no strictly semistable points, then for the forgetting map
\[
	\pi : \overline{M}_{0,n}(\PP^{d}, d)^{ss}(L') \to \overline{M}_{0,n}
\]
restricted to the semistable locus, there is no contraction on the domain curve.
\end{remark}

\subsection{Relation to Hassett's spaces}

We prove here that the morphism constructed above factors through a Hassett moduli space of weighted pointed curves.  First observe that for any linearization $(\gamma,\vec{c})\in\Delta^\circ$, the vector $\vec{c}$ defines a Hassett space $\overline{M}_{0,\vec{c}}$.

\begin{proposition}\label{prop:Hassett}
For any $(\gamma,\vec{c})\in\Delta^\circ$, there is a commutative triangle:
$$\xymatrix{ \overline{M}_{0,n} \ar[rr]^{\phi~} \ar[dr] && U_{d,n} \gquot_{(\gamma,\vec{c})} \SL(d+1) \\ & \overline{M}_{0,\vec{c}} \ar[ur] & } $$
\end{proposition}

\begin{proof}
Recall that an F-curve is an irreducible component $\overline{M}_{0,4} \hookrightarrow \overline{M}_{0,n}$ of a boundary 1-stratum, and it parameterizes a $\mathbb{P}^1$ with four ``legs'' attached; the curve is traced out by varying the cross-ratio of these attaching points.  By a result of Alexeev (cf. \cite[Lemma 4.6]{Fak09}), it is enough to show that every F-curve  contracted by the map $\overline{M}_{0,n} \rightarrow \overline{M}_{0, \vec{c}}$ is also contracted by $\phi$.  The F-curves contracted by this Hassett morphism are precisely those for which one of the tails carries $\ge c-1$ weight of marked points.  By Proposition \ref{DegreeOfTails}, for a generic linearization these F-curves are also contracted by $\phi$ because their leg carrying the most weight must have degree $d$, leaving degree zero for the component with the four attaching points.  If the linearization is not generic, then we can obtain the result by perturbing the linearization slightly: \[\overline{M}_{0,n} \rightarrow \overline{M}_{0,\vec{c}} \rightarrow \overline{M}_{0,\vec{c}-\epsilon} \rightarrow U_{d,n}\gquot_{\gamma',\vec{c}-\epsilon}\SL(d+1) \rightarrow U_{d,n}\gquot_{\gamma,\vec{c}}\SL(d+1).\]  Everything is separated and the interior $M_{0,n}$ is preserved, so this composition coincides with $\phi$.
\end{proof}

\section{Modular interpretation of chambers}\label{section:Smyth}

In the absence of strictly semistable points, i.e., for linearizations in open GIT chambers, the GIT quotients $U_{d,n} \gquot \SL(d+1)$ are fine moduli spaces of pointed rational curves.  In this section we describe explicitly the functors they represent. One approach is to describe each quotient as a moduli space of polarized pointed rational curves, as in \S\ref{ssec:moduliofpolarizedcurves}. Another useful framework for accomplishing this is provided by Smyth's notion of a \emph{modular compactification} \cite{Smy09}, cf. \S\ref{sec:modcomp}.

\subsection{GIT quotient as a moduli space of polarized curves}
\label{ssec:moduliofpolarizedcurves}

In this section, we provide a description of the GIT quotient as a moduli space of abstract genus 0 polarized curves with marked points.  This is accomplished in Theorem \ref{thm:moduli} below.  Fix $d > 0$, and let $L = (\gamma, \vec{c}) \in \Delta^{0}$ be a general linearization.

\begin{definition}\label{def:dpolarizedcurves}
Let $B$ be a noetherian scheme. A family of \textbf{$(\gamma, \vec{c})$-stable $d$-polarized curves} over $B$ consists of
\begin{itemize}
	\item a flat proper morphism $\pi : X \to B$ whose geometric fibers are
	reduced projective arithmetic genus zero curves;
	\item $n$ sections $s_{1}, \cdots, s_{n}: B \to X$;
	\item a $\pi$-ample line bundle $L$ on $X$ of degree $d$
\end{itemize}
satisfying the following numerical properties:
\begin{itemize}
	\item for $b \in B$ and a point $p \in X_{b}$ of multiplicity $m$,
	\[
		\sum_{s_{i}(b) = p}c_{i} < 1 - (m-1)\gamma;
	\]
	\item for each (possibly reducible) tail $C \subset X_{b}$,
	$\deg L|_{C} = \sigma(C)$.
\end{itemize}
\end{definition}
Here $\sigma$ is the weight function from \S\ref{section:WeightFcn}. Note that the last numerical condition is sufficient to decide the degrees of all irreducible components. 

Two families $(\pi_{1}:X_{1} \to B, \{s_{i}\}, L_{1})$, $(\pi_{2}:X_{2} \to B, \{t_{i}\}, L_{2})$ are isomorphic if there exists a $B$-isomorphism $\phi : X_{1} \to X_{2}$ such that $s_{i}\circ \phi = t_{i}$ and $\phi^{*}L_{2} \cong L_{1}\otimes \pi_{1}^{*}M$ for some line bundle $M$ over $B$.
Note that if $L$ is $\pi$-ample, then $L$ is very ample over any geometric fiber because of the genus condition. Also it is straightforward to check that $h^{0}(X_{b}, L_{b}) =  d+1$.

With a natural pull-back over base schemes, the category of families of $(\gamma, \vec{c})$-stable $d$-polarized curves forms a fibered category over the category of locally noetherian schemes.

\begin{theorem}\label{thm:moduli}
Let $\mathcal{M}_{\gamma, \vec{c}}$ be the fibered category of families of $(\gamma, \vec{c})$-stable $d$-polarized rational curves. Then $\mathcal{M}_{\gamma, \vec{c}}$ is a Deligne-Mumford stack. Moreover, it is represented by $U_{d,n}\gquot_{\gamma, \vec{c}}\SL(d+1)$.
\end{theorem}

\begin{proof}
The proof relies on standard arguments in moduli theory, so we only outline it here.

First of all, for a family of $(\gamma, \vec{c})$-stable $d$-polarized curves $\pi : X \to B$, one can show that $H^{1}(X_{b}, L_{X_{b}}) = 0$ for all geometric fibers by a straightforward induction on the number of irreducible components. Thus by \cite[Theorem III.12.11]{Har77}, $\pi_{*}L$ is locally free of rank $d+1$. By Grothendieck's descent theory, families of $(\gamma, \vec{c})$-stable $d$-polarized curves descend effectively and $\underline{\mathrm{Isom}}$ is a sheaf. Therefore $\mathcal{M}_{\gamma, \vec{c}}$ is a stack \cite[Definition 3.1]{LMB00}.

Let $\mathrm{Hilb} (1,d, \PP^{d})$ be the irreducible component of the Hilbert scheme containing rational normal curves.  Let \[HC : \mathrm{Hilb} (1,d, \PP^{d}) \to \mathrm{Chow}(1, d, \PP^{d})\] be the restricted Hilbert-Chow morphism, and let $H_{0} \subset \mathrm{Hilb}(1,d,\PP^{d})$ be the open subset parameterizing reduced non-degenerate curves.  Then the restriction $HC : H_{0} \to C_{0} := HC(H_{0})$ is injective. Moreover, there is an inverse $C_{0} \to H_{0}$, because the Hilbert polynomial of fibers of the family over $C_{0}$ is constant, so the family of algebraic cycles over $C_{0}$ is flat over $C_{0}$. Therefore $H_{0} \cong C_{0}$.

Let $U \subset \mathrm{Hilb} (1,d, \PP^{d}) \times (\PP^{d})^{n}$ be the locally closed subscheme parametrizing tuples $(X, p_1, \cdots, p_n)$ satisfying
\begin{itemize}
	\item $X \subset \PP^{d}$ is reduced, nodal and arithmetic genus zero;
	\item $p_i \in X$;
	\item $(X, \{p_i \}, \mathcal{O}_{X}(1))$ is a $(\gamma, \vec{c})$-
	stable $d$-polarized curve.
\end{itemize}
Note that for any linearization $L \in \Delta^{0}$, $U_{d, n}^{ss}(L) \subset C_{0} \times (\PP^{d})^{n}$. Also by Proposition \ref{prop:stabilitydescription}, $U \cong U_{d,n}^{ss}(L)$ within the identification $H_{0} \cong C_{0}$.

Any $(\gamma, \vec{c})$-stable $d$-polarized curve $(X, \{p_{i}\}, L)$ is represented by a point in $U$, because $L$ is very ample. Also by Proposition \ref{NoAutomorphisms}, an isomorphism between polarized curves is induced only by $\mathrm{Aut}(\PP^{d}) \cong \mathrm{PGL}(d+1)$. Therefore the map $U \to \mathcal{M}_{\gamma, \vec{c}}$ is a principal $\mathrm{PGL}(d+1)$-bundle. In particular, it is representable and faithfully flat. Moreover, the diagonal $\mathcal{M}_{\gamma, \vec{c}} \to \mathcal{M}_{\gamma, \vec{c}} \times \mathcal{M}_{\gamma, \vec{c}}$ is representable, separated and quasi-compact. By Artin's criterion (\cite[Theorem 10.1]{LMB00}), $\mathcal{M}_{\gamma, \vec{c}}$ is an algebraic stack. Moreover, since the objects have no non-trivial automorphisms, it is an algebraic space and isomorphic to its coarse moduli space.

Finally, from the above construction and the non-existence of nontrivial automorphisms,
\[
	\mathcal{M}_{\gamma, \vec{c}} \cong [U/\mathrm{PGL}(d+1)]
	\cong U/\mathrm{PGL}(d+1) \cong U_{d,n}\gquot_{L}\SL(d+1),
\]
as claimed.
\end{proof}

\subsection{Modular Compactifications}\label{sec:modcomp}

We briefly recall here the relevant results from \cite{Smy09}.  A modular compactification is defined to be an open substack of the stack of all curves that is proper over $\Spec \mathbb{Z}$ \cite[Definition 1.1]{Smy09}.  A main result of Smyth is that in genus zero these are classified by certain combinatorial gadgets.

\begin{definition}\cite[Definition 1.5]{Smy09}\label{def:extremal}
Let $\mathcal{G}$ be the set of isomorphism classes of dual graphs of strata in $\overline{M}_{0,n}$.  An \textbf{extremal assignment} $\mathcal{Z}$ is a proper (though possibly empty) subset of vertices $\mathcal{Z}(G)\subsetneq G$ for each $G\in\mathcal{G}$ such that if $G \leadsto G'$ is a specialization inducing $v \leadsto v'_1 \cup \cdots \cup v'_k$, then $v \in \mathcal{Z}(G) \Leftrightarrow v'_1,\ldots,v'_k \in \mathcal{Z}(G')$.
\end{definition}

Smyth states an additional axiom that for any $G\in\mathcal{G}$, the set $\mathcal{Z}(G)$ is invariant under $\mathrm{Aut}(G)$, but in genus zero there are no nontrivial automorphisms since $G$ is a tree with marked points on all the leaves.

\begin{definition}\cite[Definition 1.8]{Smy09}
Let $\mathcal{Z}$ be an extremal assignment.  A reduced marked curve $(X, p_1 , \ldots , p_n )$ is \textbf{$\mathcal{Z}$-stable} if there exists $(X^s , p_1^s , \ldots , p_n^s )\in\overline{M}_{0,n}$ and a surjective morphism $\pi : X^s \twoheadrightarrow X$, $\pi(p_i^s)=p_i$, with connected fibers such that:
\begin{enumerate}
\item  $\pi$ maps $X^s \backslash \mathcal{Z}( X^s )$ isomorphically onto its image, and
\item  if $X_1 , \ldots X_k$ are the irreducible components of $\mathcal{Z}( X^s )$, then $\pi ( X_i )$ is a multinodal singularity of multiplicity $\vert X_i \cap X_i^c \vert$.
\end{enumerate}
\end{definition}

The beautiful culmination of Smyth's story, in genus zero, is the following result:

\begin{theorem}[\cite{Smy09}]\label{thm:Smyth}
For any extremal assignment $\mathcal{Z}$, the stack $\overline{M}_{0,n} (\mathcal{Z})$ of $\mathcal{Z}$-stable curves is an algebraic space and a modular compactification of $M_{0,n}$.  There is a morphism $\overline{M}_{0,n} \rightarrow \overline{M}_{0,n}(\mathcal{Z})$ contracting the assigned components of each DM-stable curve.  Every modular compactification is of the form $\overline{M}_{0,n}(\mathcal{Z})$ for an extremal assignment $\mathcal{Z}$.
\end{theorem}

\subsection{Extremal assignments from GIT}

For GIT situations such that there are no strictly semistable points, the corresponding quotient is not only a categorical quotient of the semistable locus but in fact a geometric quotient \cite{git}. In the present situation, it is not hard to see that in such cases the quotient $U_{d,n}\gquot\SL(d+1)$ is a modular compactification of $M_{0,n}$ in the sense of \cite{Smy09}.  In particular, for each linearization $(\gamma,\vec{c})$ in an open GIT chamber, there is a corresponding extremal assignment.  We define here an extremal assignment $\mathcal{Z}_{\gamma , \vec{c}}$ and then show below that it is in fact the extremal assignment associated to the corresponding GIT quotient.

\begin{definition}
Let $E\subset X$ be an irreducible component of a DM-stable curve.  Set $E \in \mathcal{Z}_{\gamma , \vec{c}}(X)$ if and only if $\sum \sigma (Y) = d$,
where the sum is over all connected components $Y$ of $\overline{X \backslash E}$.
\end{definition}

\begin{proposition}\label{prop:GITass}
Let $( \gamma , \vec{c} ) \in \Delta^\circ$ be a linearization admitting no strictly semistable points.  Then $\mathcal{Z}_{\gamma , \vec{c}}$ is an extremal assignment.
\end{proposition}

\begin{proof}
It suffices to show that $\mathcal{Z} := \mathcal{Z}_{\gamma , \vec{c}}$ satisfies the axioms of Definition \ref{def:extremal}.  We first show that $\mathcal{Z}$ is invariant under specialization.  Let $v \in \mathcal{Z}(G)$, and suppose that $G \leadsto G'$ is a specialization with $v \leadsto v'_1 \cup v'_2 \cup \cdots \cup v'_k$.  To see that $v'_i \in \mathcal{Z}$ for all $i$ as well, notice that the marked points on the connected components of $G \backslash \{ v'_i \}$ contain unions of the marked points of the connected components of $G \backslash \{ v \}$.  Thus, the result follows from Lemma \ref{lem:subadd}.

Next, suppose that $v'_i \in \mathcal{Z}(G)$ for $i=1,\ldots,k$.  We must show that $v \in \mathcal{Z}(G)$ as well.  We prove this by induction on $k$, the case $k=1$ being trivial.  To prove the inductive step, let $T$ be the subtree spanned by all of the $v'_i$ and let $v'$ be a leaf of $T$.  Let $A_1 , \ldots , A_s$ denote the connected components of $G' \backslash \{ v' \}$, and let $B_1 , \ldots , B_t$ denote the connected components of $(G' \backslash T) \cup \{ v' \}$.  By assumption, $\sum_{i=1}^s \sigma ( A_i ) = d$, and by induction we may assume that $\sum_{i=1}^t \sigma ( B_i ) = d$.  Note that exactly one of the $B_i$'s contains $v'$.  Without loss of generality, we assume that this is $B_t$.  Similarly, since $v'$ is a leaf of $T$, exactly one of the $A_i$'s contains $T \backslash \{ v' \}$, and we will assume that this is $A_s$.  Note that $A_s \cup B_t = G'$, hence by additivity $\sigma ( A_s ) + \sigma ( B_t ) = d$.  It follows that $\sum_{i=1}^{s-1} \sigma ( A_i ) + \sum_{i=1}^{t-1} \sigma ( B_i ) = d$.  But the components appearing in this sum are precisely the connected components of $G' \backslash T$, and the marked points on these connected components are the same as those on the components of $G \backslash \{ v \}$.  Thus $v \in \mathcal{Z}$.

Finally, we note that $\mathcal{Z}(G)\ne G$ for each $G$, since otherwise the specialization property proved above would imply that the graph with one vertex corresponding to a smooth curve is in $\mathcal{Z}$, which is clearly not the case.
\end{proof}

Consequently, by Theorem \ref{thm:Smyth}, there is a moduli space $\overline{M}_{0,n} (\mathcal{Z}_{\gamma , \vec{c}})$ of $\mathcal{Z}_{\gamma , \vec{c}}$-stable curves and a morphism $\overline{M}_{0,n} \to \overline{M}_{0,n} (\mathcal{Z}_{\gamma , \vec{c}})$ contracting all the assigned components.

\begin{theorem}
\label{ModularDescription}
Let $( \gamma , \vec{c} ) \in \Delta^\circ$ be a linearization admitting no strictly semistable points.  Then
$$U_{d,n} \gquot_{\gamma , \vec{c}} \SL(d+1) \cong \mathcal{M}_{\gamma, \vec{c}} \cong \overline{M}_{0,n} (\mathcal{Z}_{\gamma , \vec{c}}) .$$
Moreover, a curve is GIT-stable if and only if it is $\mathcal{Z}_{\gamma , \vec{c}}$-stable.
\end{theorem}

\begin{proof}
By Theorem \ref{thm:moduli}, it suffices to prove an equivalence of the two stacks $U_{d,n}\gquot_{\gamma, \vec{c}}\SL(d+1)$ and $\overline{M}_{0,n}(\mathcal{Z}_{\gamma, \vec{c}})$.

Consider the universal family $(\pi : \mathcal{X} \hookrightarrow U_{d,n}^{ss} \times \PP^{d} \to U_{d,n}^{ss}, \{s_{i}\})$ of pointed algebraic cycles. By forgetting the embedding structure, we have a family of reduced curves. We show that each fiber is a $\mathcal{Z}_{\gamma, \vec{c}}$-stable curve, thus there is a morphism $U_{d,n}^{ss} \to \overline{M}_{0,n}(\mathcal{Z}_{\gamma, \vec{c}})$. Indeed, for a cycle $(X, p_{1}, \cdots, p_{n}) \subset \PP^{d}$ in $U_{d,n}^{ss}$, take a stable map $(f : (\widetilde{X}, p_{1}, \cdots, p_{n}) \to \PP^{d}) \in \overline{M}_{0,n}(\PP^{d}, d)^{ss}$ whose image is $(X, p_{1}, \cdots, p_{n})$. Then by Remark \ref{rem:nostabilizing}, the domain of $f$ is a stable curve. Let $\rho : \widetilde{X} \to \bar{X}$ be the $\mathcal{Z}_{\gamma, \vec{c}}$-stable contraction. For any component $E \subset \widetilde{X}$, if $\sum_{Y \subset \overline{\widetilde{X} \backslash E}} \sigma (Y) = d$ where the sum is taken for all irreducible components of $\overline{\widetilde{X}\backslash E}$, then $\rho(E)$ is a point by the definition of $\mathcal{Z}_{\gamma, \vec{c}}$.  It follows from Corollary \ref{DegreeOfComponents} that $f \vert_E$ must have degree 0 and hence $E$ is contracted by $f$.  Conversely, if $\sum_{Y \subset \overline{\widetilde{X} \backslash E}} \sigma (Y) \neq d$ (so $\rho(E)$ is not a point), then since $\deg f(\widetilde{X}) = d$, $\deg f|_{E}\neq 0$ and hence $E$ is not contracted. Therefore $\bar{X} \cong X$.

Obviously the map $U_{d,n}^{ss} \to \overline{M}_{0,n}(\mathcal{Z}_{\gamma, \vec{c}})$ is $\mathrm{PGL}(d+1)$-invariant. So we have a map $U_{d,n}\gquot_{\gamma, \vec{c}}\SL(d+1) \cong U_{d,n}^{ss}/\mathrm{PGL}(d+1) \to \overline{M}_{0,n}(\mathcal{Z}_{\gamma, \vec{c}})$.

Conversely, let $(\pi : X \to B, \{s_{i}\})$ be a family of $\mathcal{Z}_{\gamma, \vec{c}}$-stable curves. By definition of $\mathcal{Z}_{\gamma, \vec{c}}$-stability, there is a family $(\pi^{s}: X^{s} \to B, \{s_{i}^{s}\})$ of stable curves such that some of its irreducible components are contracted by the extremal assignment $\mathcal{Z}_{\gamma, \vec{c}}$. Since $\overline{M}_{0,n}(\PP^{d}, d)^{ss} \to \overline{M}_{0,n}$ is a principal $\mathrm{PGL}(d+1)$-bundle, after replacing $B$ by an \'etale covering $B' \to B$, we obtain a family of stable maps $(\pi: X^{s}\times_{B}B' \to B', f: X^{s}\times_{B}B' \to \PP^{d}, \{s_{i}\})$. By taking the image cycle, we obtain a family $(\bar{\pi} : \overline{X}\times_{B}B' \to B', \bar{f} : \overline{X}\times_{B}B' \hookrightarrow \PP^{d}\times B', \{s_{i}\})$ of pointed algebraic cycles. So we have a morphism $B' \to U_{d,n}^{ss}$. From the construction, it is easy to see that it descends to $B \to U_{d,n}^{ss}/\mathrm{PGL}(d+1) \cong U_{d,n}\gquot_{\gamma, \vec{c}}\SL(d+1)$.

We claim that this construction is independent of the choice of family $(\pi^{s} : X^{s} \to B, \{s_{i}^{s}\})$ of stable curves and hence defines a morphism $\overline{M}_{0,n}(\mathcal{Z}_{\gamma, \vec{c}}) \to U_{d,n}\gquot_{\gamma, \vec{c}}\SL(d+1)$. To see this, we need to check that the contracted component of $X^{s}$ by $\mathcal{Z}_{\gamma, \vec{c}}$-stability is also contracted by the cycle map. The computation is identical to the previous one.

It is straightforward to see that the two morphisms constructed above give an equivalence of categories between $\overline{M}_{0,n}(\mathcal{Z}_{\gamma, \vec{c}})$ and $U_{d,n}\gquot_{\gamma, \vec{c}}SL(d+1) \cong \mathcal{M}_{\gamma, \vec{c}}$.
\end{proof}

\section{Maps Between Moduli Spaces}\label{section:maps}

In this section we describe maps between the various different quotients of $U_{d,n}$.  The gluing maps are related to known maps defined on $\overline{M}_{0,n}$.  The projection and VGIT maps, on the other hand, form a large set of explicit maps that do not appear previously in the literature.

\subsection{Gluing Maps}

The first maps we consider are helpful for understanding the boundary of these moduli spaces.  Recall that each of the boundary divisors in $\overline{M}_{0,n}$ corresponds to a subset $I \subset [n]$ with $\vert I \vert = i, 2 \leq i \leq \frac{n}{2}$.  Each such divisor $D_I$ is the image of a gluing map:
$$ \overline{M}_{0,i+1} \times \overline{M}_{0,n-i+1} \to \overline{M}_{0,n} .$$
In this section we describe a natural analogue of these gluing maps for the GIT quotients $U_{d,n} \gquot_{\gamma,\vec{c}} \SL(d+1) \cong \overline{M}_{0,n}(\mathcal{Z}_{\gamma,\vec{c}})$.

\begin{proposition}
\label{GluingMap}
Let $( \gamma , \vec{c} ) \in \Delta^\circ$ be such that there are no strictly semistable points, and let $I \subset [n]$ be a subset such that $\sigma (I) \neq 0,d$ and write $i = \vert I \vert$.  We write $\vec{c}_I$ for the vector consisting of the weights $c_i$ for all $i \in I$.  Then there is a ``gluing'' morphism $\Gamma_{i}$ such that the following diagram commutes:
$$\xymatrix{
\overline{M}_{0,i+1} \times \overline{M}_{0,n-i+1} \ar[r] \ar[d] & \overline{M}_{0,n} \ar[d] \\
\overline{M}_{0,i+1} (\mathcal{Z}_{\gamma , \vec{c}_I , b_I} ) \times \overline{M}_{0,n-i+1} ( \mathcal{Z}_{\gamma , \vec{c}_{I^c} , b_{I^c}} ) \ar[r]^{\phantom{aaaaaaaaaaaaaa}\Gamma_{i}} & \overline{M}_{0,n} ( \mathcal{Z}_{\gamma , \vec{c}} ) }$$
where $b_I = (1- \gamma ) \sigma (I) - ( c_{I}- 1) + \gamma$.  Similarly, if $\sigma (I) = d$, then there is a commutative diagram:
$$\xymatrix{
\overline{M}_{0,i+1} \times \overline{M}_{0,n-i+1} \ar[r] \ar[d] & \overline{M}_{0,n} \ar[d] \\
\overline{M}_{0,i+1} ( \mathcal{Z}_{\gamma , \vec{c}_I , b_I} ) \ar[r]^{\phantom{aaa}\Gamma} & \overline{M}_{0,n} ( \mathcal{Z}_{\gamma , \vec{c}} ) .}$$
Moreover, the horizontal maps are all injective.
\end{proposition}

\begin{proof}
First of all, we prove the existence of $\Gamma_{i}$. By using Theorem \ref{ModularDescription}, let $\mathcal{M}_{\gamma, \vec{c}_{I}, b_{I}} := \overline{M}_{0, i+1}(\mathcal{Z}_{\gamma, \vec{c}_{I}, b_{I}})$ and let $\mathcal{M}_{\gamma, \vec{c}_{I^{c}}, b_{I^{c}}} := \overline{M}_{0, n-i+1}(\mathcal{Z}_{\gamma, \vec{c}_{I^{c}}, b_{I^{c}}})$. For a base scheme $B$, let $(\pi_{1} : X_{1} \to B, \{s_{j}, p\}, L_{1})$ (resp. $(\pi_{2} : X_{2} \to B, \{t_{k}, q\}, L_{2})$) be a family of $(\gamma, \vec{c}_{I}, b_{I})$-stable $d_{1}$-polarized curves (resp. $(\gamma, \vec{c}_{I^{c}}, b_{I^{c}})$-stable $d_{2}$-polarized curves). Note that the gluing of two schemes along isomorphic closed subschemes always exists in the category of schemes. So we can glue $X_{1}$ and $X_{2}$ along two isomorphic sections $p$ and $q$, and obtain $X$. Since we glued along sections, there is a morphism $\pi : X \to B$ and sections $\{s_{j}, t_{k}: B \to X\}$. Finally, two line bundles $L_{1}$ and $L_{2}$ also can be glued if we consider them as $\mathbb{A}^{1}$-fibrations over $X_{1}$ and $X_{2}$. So over $X$, there is a line bundle $L$ which is of degree $d:=d_{1}+d_{2}$ over each fiber of $\pi$. This is a flat family, since the Hilbert polynomials of fibers are constant. This construction is functorial, thus we have a morphism of stacks from $\mathcal{M}_{\gamma, \vec{c}_{I}, b_{I}} \times \mathcal{M}_{\gamma, \vec{c}_{I^{c}}, b_{I}^{c}}$ to the stack of $n$-pointed genus zero curves. 

Now we need to show that the glued family $(\pi : X = X_{1}\cup_{p=q}X_{2} \to B, \{p_{i}\}:=\{s_{j},t_{k}\}, L)$ is in $\mathcal{M}_{\gamma, \vec{c}} \cong \overline{M}_{0,n}(\mathcal{Z}_{\gamma, \vec{c}})$. It suffices to check this fiberwise. So we may assume that $B$ is a closed point. For a point $x \in X$, if it is not the gluing point, then 
\[
	\sum_{p_{i}=x}c_{i} < 1-(m-1)\gamma
\]
is immediate. If $x$ is the gluing point of $p$ and $q$ of multiplicity $m_{1}$ and $m_{2}$ respectively, 
\[
	\sum_{p_{i}=x}c_{i} = \sum_{s_{j}=p}c_{i} + \sum_{t_{k}=q}c_{i}
	< 1 - (m_{1}-1)\gamma - b_{I} + 1 - (m_{2}-1)\gamma - b_{I^{c}}
\]
\[
	= 1 - (m_{1}+m_{2}-1)\gamma.
\]
Since the multiplicity of $x$ in $X$ is $m_{1}+m_{2}$, it satisfies the first numerical condition in Definition \ref{def:dpolarizedcurves}.

Next, since $X$ is a gluing of two curves at one point, for a tail $Y$, $Y$ or its complement tail $\overline{X\backslash Y}$ is contained in one of $X_{1}$ or $X_{2}$. If $Y = X_{1}$ (so $\overline{X\backslash Y} = X_{2}$), then $\deg L|_{X_{1}} = \deg L_{1}|_{X_{1}} = \lceil \frac{c_{I}+b_{I}-1}{1-\gamma}\rceil \ge \lceil \frac{c_{I}-1}{1-\gamma}\rceil = \sigma(X_{1}) = d_{1}$. By the same idea, $\deg L|_{X_{2}} \ge \sigma(X_{2}) = d_{2}$. Now since $d_{1}+d_{2}=d = \deg L|_{X_{1}}+\deg L|_{X_{2}}$, $\deg L|_{X_{1}} = \sigma(X_{1})$ and $\deg L|_{X_{2}} = \sigma(X_{2})$. 

If $Y$ is a proper subset of $X_{1}$, then $\deg L|_{Y} = \deg L_{1}|_{Y} = \sigma(Y)$ because $\sigma(Y)$ depends only on $\{c_{i}\}_{p_{i} \in Y}$ and $\gamma$, not on $d_{1}$ or $d$. Finally if $\overline{X\backslash Y}$ is a proper subset of $X_{1}$, then 
\[
	\deg L|_{Y} = d - \deg L|_{\overline{X\backslash Y}} = 
	d - \deg L_{1}|_{\overline{X\backslash Y}} = 
	d - \sigma(\overline{X \backslash Y}) = \sigma(Y).
\]
Note that the last equality holds because the numerical data $(\gamma, \vec{c})$ satisfies the normalization condition $(d-1)\gamma + \sum c_{i} = d +1$, hence the additivity lemma (Lemma \ref{Additivity}) holds. Therefore all tails have correct degrees. So it is in $\mathcal{M}_{\gamma, \vec{c}}$. 

Having proven the existence of the gluing morphism, to check commutativity of the diagram is straightforward. We leave the simpler case $\sigma(I) = d$ to the reader. 
\end{proof}

\begin{remark}
\label{Normality}
We would like to conclude more strongly that the gluing maps are all embeddings, which would follow if the varieties in question were all normal.  Several of the results below about maps between these GIT quotients could be similarly strengthened using normality.  We note here that, since the map $\overline{M}_{0,n} \to \overline{M}_{0,n} ( Z_{\gamma, \vec{c}} )$ has connected fibers, the normalization map $\overline{M}_{0,n} ( Z_{\gamma, \vec{c}} )^{\nu} \to \overline{M}_{0,n} ( Z_{\gamma, \vec{c}} )$ (equivalently, $U_{d,n}\gquot_{\gamma, \vec{c}}\SL(d+1)^{\nu} \to U_{d,n}\gquot_{\gamma, \vec{c}}\SL(d+1)$) is bijective.  Although we strongly suspect that it is indeed an isomorphism, at present we have no proof.
\end{remark}

\subsection{Projection Maps}

Another natural set of maps between these moduli spaces is given by projection from the marked points.

\begin{proposition}
Let $( \gamma , \vec{c} ) \in \Delta^\circ$ be such that there are no strictly semistable points, and suppose that $d \geq 2$ and $c_1 > 1- \gamma$.  Then projection from $p_1$ defines a birational morphism
$$ \pi_i : U_{d,n} \gquot_{( \gamma , \vec{c} )} \SL(d+1) \to U_{d-1,n} \gquot_{( \gamma , c_1 - (1 - \gamma ) , c_2 \ldots , c_n )} \SL(d) .$$
\end{proposition}

\begin{proof}
First, note that since $c_1 > 1 - \gamma$, every GIT-stable curve is smooth at $p_1$ by Corollary \ref{SmoothWeight}. It follows that, if $(X, p_1 , \ldots , p_n )$ is a GIT-stable curve, then its projection $\pi_{p_1} (X, p_1 , \ldots , p_n )$ is a connected rational curve of degree $d-1$ in $\PP^{d-1}$.  We show that this projected curve is stable for the linearization $( \gamma , c_1 - (1 - \gamma ) , c_2 \ldots , c_n )$ if and only if the original curve is stable for the linearization $( \gamma , c_1 , \ldots , c_n )$.  Indeed, every component of $\pi_{p_1} (X)$ has the same degree as its preimage, unless its preimage contains $p_1$, in which case the degree drops by one.  It follows that, for any tail $Y \subset \pi_{p_1} (X)$, we have
\begin{displaymath}
\deg (Y) = \left\{ \begin{array}{ll}
\lceil \frac{ ( \sum_{p_i \in Y} c_i ) - 1}{1 - \gamma} \rceil & \textrm{if $p_1 \notin Y$}\\
\lceil \frac{ ( \sum_{p_i \in Y} c_i ) - (1 - \gamma ) - 1}{1 - \gamma} \rceil & \textrm{if $p_1 \in Y$}
\end{array} \right.
\end{displaymath}
But this is exactly the condition for stability of points in $U_{d-1,n}$ for the linearization $( \gamma , c_1 - (1 - \gamma ) , c_2 \ldots , c_n ))$.
\end{proof}

\begin{proposition}
\label{ProjectionIsomorphism}
The projection map $\pi_1$ is a bijective morphism if and only if, for every partition $\{2, \ldots , n \} = I_1 \sqcup \cdots \sqcup I_k$ into at least 3 disjoint sets, we have $\sum_{i=1}^k \sigma ( I_i ) \neq d-1$.
\end{proposition}

\begin{proof}
Let $E \subset X$ be a component of a GIT-stable curve with respect to the linearization $( \gamma , \vec{c} )$.  $E$ is contracted by the projection map if and only if $p_1 \in E$ and deg $E=1$.  It follows that the map is bijective if and only if every such component has no moduli, which is equivalent to every such component having exactly three special points, where here a ``special point'' is either a singular point (regardless of the singularity type) or a marked point (regardless of how many of the $p_i$'s collide at that point).  By Corollary \ref{DegreeOfComponents}, we therefore see that $\pi_1$ is a bijective morphism if and only if the hypothesis holds.
\end{proof}

\subsection{Wall-Crossing Maps}

One of the benefits of our GIT approach is that, by varying the choice of linearization, we obtain explicit maps between our moduli spaces.  The nature of these maps can be understood using the general theory of variation of GIT.

Recall that, by Proposition \ref{GITWalls}, the GIT walls in $\Delta^\circ$ are of the form $\varphi (I, \cdot )^{-1} (k)$ for any given subset $I \subset [n]$ and integer $k$.  For a fixed such $I$ and $k$, we let $( \gamma , \vec{c} ) \in \varphi (I, \cdot )^{-1} (k) = \varphi (I^c , \cdot )^{-1} (d-1-k)$ be such that $( \gamma , \vec{c} )$ does not lie on any other walls, and we write
$$ U_{d,n} \gquot_{\gamma,\vec{c},0} \SL(d+1) := U_{d,n} \gquot_{\gamma , \vec{c}} \SL(d+1) .$$
Similarly, we will write $U_{d,n} \gquot_{\gamma,\vec{c},+} \SL(d+1)$ and $U_{d,n} \gquot_{\gamma,\vec{c},-} \SL(d+1)$ for the GIT quotients corresponding to the neighboring chambers, which are contained in $\varphi (I, \cdot )^{-1} ( \{ x>k \} )$ and $\varphi (I, \cdot )^{-1} ( \{ x<k \} )$, respectively.  We will write $\sigma_+$, $\sigma_-$ for the $\sigma$ functions on either side of the wall.  Note that, for any subset $A \subset [n]$, $\sigma_+ (A) = \sigma_- (A)$ if and only if $A \neq I, I^c$.  By general VGIT, there is a commutative diagram:
$$\xymatrix{
U_{d,n} \gquot_{\gamma,\vec{c},+} \SL(d+1) \ar@{<.>}[rr] \ar[rd] & & U_{d,n} \gquot_{\gamma,\vec{c},-} \SL(d+1) \ar[ld] \\
 & U_{d,n} \gquot_{\gamma,\vec{c},0} \SL(d+1) & . }$$

We now consider stability conditions at a wall.  For these linearizations, a new type of semistable curve appears:

\begin{definition}
A pointed curve $(X, p_1 , \ldots , p_n ) \in U_{d,n}$ is a \textbf{$(\gamma,\vec{c})$-bridge} if:
\begin{enumerate}
\item  $X$ has a degree 1 component $D$ such that $\vert D \cap \overline{X \backslash D} \vert = 2$;
\item  If we write $X_{I}, X_{I^c}$ for the connected components of $\overline{X \backslash D}$, then $X_I$ is marked by the points in $I$ and $X_{I^c}$ is marked by the points in $I^c$;
\item  If $E \subset X_I$ (resp. $X_{I^c}$) is a connected subcurve, then the degree of $E$ is equal to $d - \sum_Y \sigma_- (Y)$ (resp. $d - \sum_Y \sigma_+ (Y)$), where the sum is over all connected components of $\overline{X \backslash E}$.
\end{enumerate}
\end{definition}

Note that, by definition, $\deg ( X_I ) = k$ and $\deg ( X_{I^c} ) = d-(k+1)$, as in the following picture:

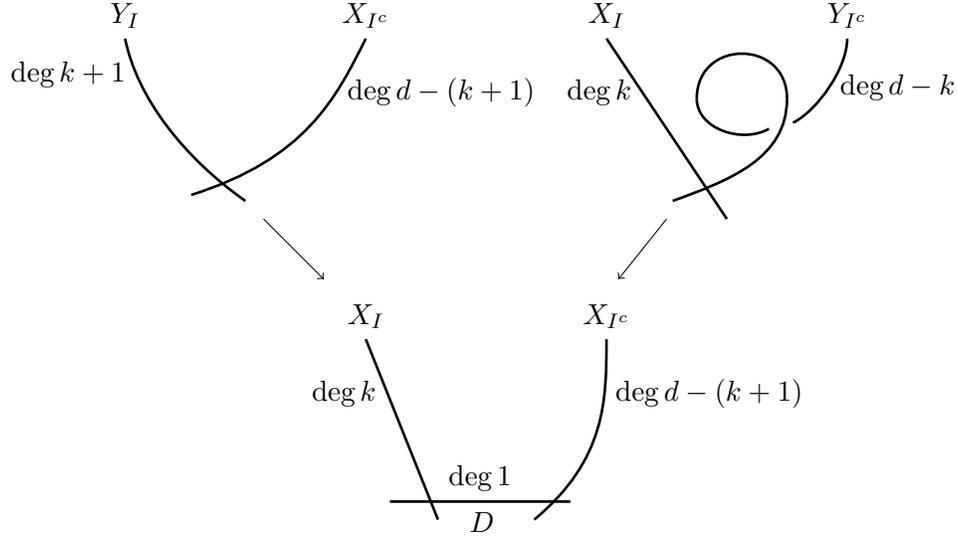
\begin{center}
\small
\begin{tikzpicture}[scale=0.8]
	\draw[line width=1pt] (0, 10) .. controls(0.2, 9) and (1, 8) ..
	(2, 7.3)
	node[pos=0, above] {$Y_{I}$}
	node[pos=0.2, left] {$\deg k+1$};
	\draw[line width=1pt] (1.1, 7.4) .. controls(3, 8) and (3.5, 9) ..
	(4, 10)
	node[pos=1, above] {$X_{I^{c}}$}
	node[pos=0.7, right] {$\deg d-(k+1)$};
	
	\draw[line width=1pt] (8, 10) -- (10, 7)
	node[pos=0, above] {$X_{I}$}
	node[pos=0.3, left] {$\deg k$};	
	\draw[line width=1pt] (9.1, 7.3) to [out=20, in=-90] (11, 9);
	\draw[line width=1pt] (11, 9) .. controls(11, 10) and (9.5, 10) ..
	(9.5, 9);
	\draw[line width=1pt] (9.5, 9) to [out=-90, in=210] (10.7, 8.5);
	\draw[line width=1pt] (11.1, 8.6)
	.. controls(11.5, 8.8) and (12, 9.5) .. (12, 10)
	node[pos=1, above] {$Y_{I^{c}}$}
	node[pos=0.5, right] {$\deg d- k$};
	
	\draw[line width=1pt] (4, 5) -- (5.2, 2)
	node[pos=0, above] {$X_{I}$}
	node[pos=0.3, left] {$\deg k$};
	\draw[line width=1pt] (4.4, 2.3) -- (7.4, 2.3)
	node[pos=0.5, above] {$\deg 1$}
	node[pos=0.51, below] {$D$};
	\draw[line width=1pt] (6.8, 2) .. controls(8, 3) and (8, 4) ..(8, 5)
	node[pos=0.7, right] {$\deg d-(k+1)$}
	node[pos=1, above] {$X_{I^{c}}$};
	
	\draw[->] (2.3, 7) -- (3.3, 6);
	\draw[->] (9, 7) -- (8.2, 6);
\end{tikzpicture}
\normalsize
\end{center}

\begin{proposition}
\label{FlatLimit}
Every $(\gamma,\vec{c})$-bridge is GIT-semistable at the wall $\varphi (I, \cdot )^{-1} (k)$.
\end{proposition}

\begin{proof}
Let $(X, p_1 , \ldots , p_n )$ be a $(\gamma,\vec{c})$-bridge.  It suffices to construct a $(\gamma,\vec{c},+)$-stable curve $(Y , q_1 , \ldots , q_n )$ and a 1-PS $\lambda$ such that $$\mu_{\lambda} (Y, q_1 , \ldots , q_n ) = 0 \text{    and} $$
$$ \lim_{t \to 0} \lambda (t) \cdot (Y, q_1 , \ldots , q_n ) = (X, p_1 , \ldots , p_n ) .$$
Let $(X_I , p_1 , \ldots p_m , p)$ denote the tail of $X$ labeled by points in $I$, where $p$ is the ``attaching point''.  Note that, by Proposition \ref{GluingMap} and the fact that $( \gamma , \vec{c} )$ does not lie on any walls other than $\varphi (I, \cdot )^{-1} (k)$, $X_I$ is stable for the linearization $(\gamma ,c_1,\ldots,c_m,\gamma - \epsilon )$.  Because the projection map is proper and birational, there is a curve $(Y_I,q_1,\ldots ,q_m,q)$, stable for the linearization $(\gamma ,c_1,\ldots ,c_m,1-\epsilon )$, such that $\pi_q ( Y_I ) = X_I$.

Choose coordinates so that the span of $Y_I$ is $V( x_{k+2} , \ldots , x_d )$ and $q = V( x_0 , \ldots x_k , \widehat{x_{k+1}} , x_{k+2} , \ldots , x_d )$.  Now, let $\lambda$ be the 1-PS that acts with weights $(0, \ldots , 0,1, \ldots ,1)$, where the first $k+1$ weights are all zero.  Let $i : Y_I \hookrightarrow \PP^d$ be the inclusion and consider the rational map
$$ U := \mathbb{C} \times Y_I \dashrightarrow \PP^d $$
given by $(t,r) \mapsto \lambda (t) \cdot i(r)$.  Note that this map is regular everywhere except the point $(0,q)$.  If we blow up $U$ at this point, we obtain a regular map $ \tilde{U} \to \PP^d $ whose special fiber is the union of $\pi_q ( Y_I ) = X_I$ and a line.  Since the image of the point $q$ is constant in this family, we may glue on $X_{I^c}$ to obtain a family of connected degree $d$ curves.  By Proposition \ref{GluingMap}, $Y = Y_I \cup X_{I^c}$ is a $(\gamma,\vec{c},+)$-stable curve.  Note that, since $(Y, q_1 , \ldots , q_n )$ is $(\gamma , \vec{c})$-semistable, but its limit under the 1-PS $\lambda$ is not isomorphic to itself, we must have $\mu_{\lambda} (Y, q_1 , \ldots , q_n ) = 0$.  It follows that $(X, p_1 , \ldots , p_n )$ is semistable.
\end{proof}

We will see that the $(\gamma,\vec{c})$-bridges are the only ``new'' curves that appear at the wall.

\begin{proposition}
\label{WallCrossing}
A pointed curve $(X, p_1 , \ldots , p_n ) \in U_{d,n}$ is stable for the linearization $(\gamma,\vec{c},0)$ if and only if it is stable for the linearization $(\gamma,\vec{c},+)$ (equivalently, $(\gamma,\vec{c},-)$) and does not contain a tail labeled by the points in $I$ or $I^c$.  It is strictly semistable if and only if it contains a tail labeled by the points in $I$ or $I^c$, and is either $(\gamma,\vec{c},+)$-stable, $(\gamma,\vec{c},-)$-stable, or a $(\gamma,\vec{c})$-bridge.  Moreover, the $(\gamma,\vec{c})$-bridges are exactly the strictly semistable curves with closed orbits.
\end{proposition}

\begin{proof}
We first show that each of the curves above is (semi)stable.  It is a standard fact from variation of GIT that, if a curve is stable for both linearizations $(\gamma,\vec{c},+)$ and $(\gamma,\vec{c},-)$, then it is stable for the linearization $(\gamma,\vec{c},0)$ as well.  By assumption, the only wall that $( \gamma , \vec{c} )$ lies on is $\varphi (I, \cdot )^{-1} (k) = \varphi (I^c , \cdot )^{-1} (d-1-k)$, so any curve that does not contain a tail labeled by the points in $I$ will be stable for one of these linearizations if and only if it is stable for the other.  Similarly, if a curve is stable for either linearization $(\gamma,\vec{c},+)$ or $(\gamma,\vec{c},-)$, then it is semistable for the linearization $(\gamma,\vec{c},0)$.  It therefore suffices to show that $(\gamma,\vec{c})$-bridges are GIT-semistable, but this was shown in Proposition \ref{FlatLimit}.

To see the converse, let $(X, p_1 , \ldots , p_n ) \in U_{d,n}$ be semistable for the linearization $(\gamma,\vec{c},0)$.  Notice that the degree of each tail $Y \subset X$ is completely determined by $\sigma$ unless $Y$ is labeled by points in $I$ or $I^c$.  We therefore see that, if $X$ contains no tails labeled by points in $I$ or $I^c$, then for any connected subcurve $E \subset X$ we have
$$ \deg (E) = d - \sum \sigma (Y) $$
and $(X, p_1 , \ldots , p_n )$ is a $(\gamma,\vec{c},+)$-stable curve.

Similarly, suppose that $X$ contains a subcurve $E$ such that $\overline{X \backslash E}$ contains a connected component $X_I$ labeled by $I$ but no connected component labeled by $I^c$.  Then the degree of $X_I$ is either $k$ or $k+1$, and thus either
$$ \deg (E) = d-k - \sum \sigma (Y) $$
or
$$ \deg (E) = d-(k+1) - \sum \sigma (Y) $$
where the sum is over all connected components $Y \subset \overline{X \backslash E}$ other than $X_I$.  It follows that $(X, p_1 , \ldots , p_n )$ is either $(\gamma,\vec{c},+)$-stable or $(\gamma,\vec{c},-)$-stable.

The remaining case is where $X$ contains a component $E$ such that $\overline{X \backslash E}$ contains a connected component $X_I$ labeled by the points in $I$ and a connected component $X_{I^c}$ labeled by the points in $I^c$.  Since $\deg X_I \geq k+1, \deg X_{I^c} \geq d-(k+1)$, and $\deg E \geq 1$, we see that the only possibility is if all three inequalities hold.  Thus, $E$ is a degree 1 subcurve of $X$ such that $\vert E \cap \overline{X \backslash E} \vert = 2$, and $(X, p_1 , \ldots , p_n )$ is a $(\gamma,\vec{c})$-bridge.

Finally, note that if a semistable curve does not have a closed orbit, then it degenerates to a semistable curve with higher-dimensional stabilizer.  Furthermore, a strictly semistable curve with closed orbit cannot have a 0-dimensional stabilizer.  Since $(\gamma,\vec{c})$-bridges have 1-dimensional stabilizers and all other semistable curves have 0-dimensional stabilizers, we see that the $(\gamma,\vec{c})$-bridges must be precisely the strictly semistable curves with closed orbits.
\end{proof}

We can restate the results of Proposition \ref{WallCrossing} in the following way.  Each of the maps in the diagram
$$\xymatrix{
U_{d,n} \gquot_{\gamma,\vec{c},+} \SL(d+1) \ar@{<.>}[rr] \ar[rd] & & U_{d,n} \gquot_{\gamma,\vec{c},-} \SL(d+1) \ar[ld] \\
 & U_{d,n} \gquot_{\gamma,\vec{c},0} \SL(d+1) &  }$$
restricts to an isomorphism away from the image of $D_I \subset \overline{M}_{0,n}$.  If $k \neq 0,d-1$, then along the image of this divisor, the maps restrict to the following:
$$\xymatrix{
\overline{M}_{0,i+1} ( \mathcal{Z}_{\gamma , \vec{c}_I , 1- \epsilon} ) \times \overline{M}_{0,n-i+1} ( \mathcal{Z}_{\gamma , \vec{c}_{I^c} , \gamma + \epsilon}) \ar[r]^{\phantom{aaaaaaaaaa}\Gamma_i} \ar[d]^{(\pi_{i+1} , id)} & U_{d,n} \gquot_{\gamma,\vec{c},+} \SL(d+1) \ar[d] \\
\overline{M}_{0,i+1} ( \mathcal{Z}_{\gamma , \vec{c}_I , \gamma + \epsilon} ) \times \overline{M}_{0,n-i+1} ( \mathcal{Z}_{\gamma , \vec{c}_{I^c} , \gamma + \epsilon}) \ar[r] & U_{d,n} \gquot_{\gamma,\vec{c},0} \SL(d+1) \\
\overline{M}_{0,i+1} ( \mathcal{Z}_{\gamma , \vec{c}_I , \gamma + \epsilon} ) \times \overline{M}_{0,n-i+1} ( \mathcal{Z}_{\gamma , \vec{c}_{I^c} , 1 - \epsilon}) \ar[r]^{\phantom{aaaaaaaaaa}\Gamma_i} \ar[u]_{(id,\pi_{n-i+1})} & U_{d,n} \gquot_{\gamma,\vec{c},-} \SL(d+1) \ar[u] }$$
where the central map is obtained by gluing a line between the attaching points.

Similarly, if $k=d-1$, the maps restrict to:
$$\xymatrix{
\overline{M}_{0,i+1} ( \mathcal{Z}_{\gamma , \vec{c}_I , 1- \epsilon} ) \ar[r]^{\Gamma_i} \ar[d]^{\pi_{i+1}} & U_{d,n} \gquot_{\gamma,\vec{c},+} \SL(d+1) \ar[d] \\
\overline{M}_{0,i+1} ( \mathcal{Z}_{\gamma , \vec{c}_I , \gamma + \epsilon} ) \ar[r] & U_{d,n} \gquot_{\gamma,\vec{c},0} \SL(d+1) \\
\overline{M}_{0,i+1} ( \mathcal{Z}_{\gamma , \vec{c}_I , \gamma + \epsilon} ) \times \overline{M}_{0,n-i+1} ( \mathcal{Z}_{\gamma , \vec{c}_{I^c} , 1 - \epsilon}) \ar[r]^{\phantom{aaaaaaaaaa}\Gamma_i} \ar[u]_{(id,\cdot)} & U_{d,n} \gquot_{\gamma,\vec{c},-} \SL(d+1). \ar[u] }$$

\subsection{Quotients at the Boundary of $\Delta^\circ$}\label{section:BoundaryQuots}

There are four distinct types of top-dimensional boundary walls, corresponding to when $\gamma = 0$, $\gamma = 1$, $c_i = 0$ for some $i$, and $c_i = 1$ for some $i$.  In this section, we consider each in turn.

\begin{corollary}
\label{OnePointWall}
Suppose $c_1 = 1 - \epsilon$ for $\epsilon \ll 1$.  Then after replacing GIT quotients by their normalizations, the map induced by passing to the GIT wall $c_1 = 1$ is a projection map:
$$\xymatrix{
U_{d,n} \gquot_{\gamma , c_1 , \ldots , c_n} \SL(d+1) \ar[d]_{f} \ar[rd]^{\pi_1} & \\
U_{d,n} \gquot_{\gamma - \frac{\epsilon}{d-1} , 1, c_2 , \ldots , c_n} \SL(d+1) \ar[r]^{\cong}_g & U_{d-1,n} \gquot_{\gamma , \gamma - \epsilon, c_2 , \ldots , c_n} \SL(d)  }.$$
\end{corollary}

\begin{proof}
If we replace all GIT quotients by their normalizations, the morphisms between them form algebraic fiber spaces. In particular, we can apply the rigidity lemma (\cite[Proposition II.5.3]{Kol96}). 

Note that the boundary wall $c_1 = 1$ is equal to the hyperplane $\varphi (\{ 1 \}, \cdot )^{-1} (0)$.  Let $X$ be a $( \gamma , \vec{c} ,-)$-stable curve.  By Proposition \ref{FlatLimit}, we see that there is a $( \gamma , \vec{c} , 0)$-semistable curve with closed orbit consisting of the projected curve $\pi_1 (X)$ together with a degree 1 tail $L$ containing $p_1$ and attached at $\pi ( p_1 )$. Conversely, all $(\gamma, \vec{c}, 0)$-semistable curves with closed orbits are of this form. Thus for such a curve $Y$, the fiber $f^{-1}(Y)$ is positive-dimensional if and only if $X \in f^{-1}(Y)$ has a unique irreducible tail of degree 1 containing $p_{1}$ and at least two more marked points on its smooth locus. Now it is easy to see that $f^{-1}(Y)$ is contracted by $\pi_{1}$. Therefore, by the rigidity lemma we have a morphism $g : U_{d,n}\gquot_{\gamma-\frac{\epsilon}{d-1}, 1, c_{2}, \cdots, c_{n}}\SL(d+1) \to U_{d-1,n}\gquot_{\gamma, \gamma-\epsilon, c_{2}, \cdots, c_{n}}\SL(d)$. 

Since the points of the GIT quotient are in bijection with the closed orbits of semistable points, it is straightforward to check that the induced horizontal map is bijective and indeed an isomorphism.
\end{proof}

\begin{proposition}
\label{GammaEqualsOne}
When $\gamma = 1$, we have the following isomorphism:
$$ U_{d,n} \gquot_{1,\vec{c}} \SL(d+1) \cong ( \PP^1 )^n \gquot_{\vec{c}} \SL(2) .$$
\end{proposition}

\begin{proof}
By Corollary \ref{cor:smoothcurves}, every GIT-stable curve is smooth, and by Corollary \ref{SmoothWeight}, at most half of the total weight may collide at a marked point.  We therefore have a map
$$ U_{d,n} \gquot_{1,\vec{c}} \SL(d+1) \to ( \PP^1 )^n \gquot_{\vec{c}} \SL(2) .$$
On the other hand, note that Kapranov's morphism $\overline{M}_{0,n} \to \overline{M}_{0,\vec{c}+\epsilon} \to (\PP^{1})^{n}\gquot_{\vec{c}}\SL(2)$ is a composition of divisorial contractions. Thus one may run the same argument as in Proposition \ref{prop:Hassett} (and in \cite[Lemma 4.6]{Fak09}), thinking of $( \PP^1 )^n \gquot_{\vec{c}} \SL(2)$ as an analogue of the Hassett space $\overline{M}_{0, \vec{c} }$ when $\sum_{i=1}^n c_i = 2$, to see that there is a map $f: ( \PP^1 )^n \gquot_{\vec{c}} \SL(2) \to U_{d,n} \gquot_{1,\vec{c}} \SL(d+1)$.
\end{proof}

With Corollary \ref{OnePointWall} and Proposition \ref{GammaEqualsOne}, we now have a complete description of all of the boundary walls of the GIT cone $\Delta^\circ$.  If $c_i = 1$ for some $i$, then the corresponding map is a projection map.  If $\gamma = 1$, then the quotient is isomorphic to $( \PP^1 )^n \gquot_{\vec{c}} \SL(2)$.  On the other hand, if $c_i = 0$ for some $i$, then the corresponding map is a forgetful map, whereas if $\gamma = 0$, the quotient is isomorphic to the spaces $V_{d,n} \gquot_{\vec{c}} \SL(d+1)$ studied in \cite{NoahCB}.\footnote{In the latter two statements, the line bundles in question are only semi-ample rather than ample, and hence by Mumford's definition the corresponding GIT quotients are quasi-projective rather than projective.  If, however, one defines the GIT quotient to be Proj of the invariant section ring, then these statements are fine.}

\subsection{Behavior of Wall-Crossing Maps}\label{section:wallcrossing}

By the above diagram, we also have a nice description of wall-crossing behavior along the interior walls.

\begin{corollary}
\label{DivisorialContraction}
The morphism
$$U_{d,n} \gquot_{\gamma,\vec{c},+} \SL(d+1) \rightarrow U_{d,n} \gquot_{\gamma,\vec{c},0} \SL(d+1)$$
contracts a divisor if and only if $3 \leq \vert I \vert \leq n-2$ and $k=0$.  Similarly, the morphism
$$U_{d,n} \gquot_{\gamma,\vec{c},-} \SL(d+1) \rightarrow U_{d,n} \gquot_{\gamma,\vec{c},0} \SL(d+1)$$
contracts a divisor if and only if $2 \leq \vert I \vert \leq n-3$ and $k=d-1$.
\end{corollary}

\begin{proof}
This follows directly from the diagram above.  Because the map restricts to an isomorphism away from the image of $D_{I,I^c} \subset \overline{M}_{0,n}$, the only divisor that could be contracted by the map is the image of this divisor. In the diagram above, however, which details the restriction of this map to this divisor, all of the restricted maps are birational unless $k=0$ and $3 \leq \vert I \vert \leq n-2$ or $k=d-1$ and $2 \leq \vert I \vert \leq n-3$.
\end{proof}

\begin{corollary}
\label{Flips}
If $k \neq 0, d-1$, then the rational map
$$U_{d,n} \gquot_{\gamma,\vec{c},+} \SL(d+1) \dashrightarrow U_{d,n} \gquot_{\gamma,\vec{c},-} \SL(d+1)$$
either induces a morphism on the normalizations, its inverse induces a morphism on the normalizations, or it is a flip.
\end{corollary}

\begin{proof}
Consider the diagram:
$$\xymatrix{
U_{d,n} \gquot_{\gamma,\vec{c},+} \SL(d+1) \ar@{<.>}[rr] \ar[rd]^{f^+} & & U_{d,n} \gquot_{\gamma,\vec{c},-} \SL(d+1) \ar[ld]_{f^-} \\
 & U_{d,n} \gquot_{\gamma,\vec{c},0} \SL(d+1) & . }$$
The result follows from \cite[Theorem 3.3]{Thaddeus}, since if neither $f^+$ nor $f^-$ is bijective then both are small contractions, by the gluing diagram above.
\end{proof}

Note that this is a flip in the sense of \cite{Thaddeus}.  That is, there exists a $\mathbb{Q}$-Cartier divisor class $D$ on $U_{d,n} \gquot_{\gamma,\vec{c},-} \SL(d+1)$, such that $\mathcal{O} (-D)$ is relatively ample over $U_{d,n} \gquot_{\gamma,\vec{c},0} \SL(d+1)$, and if $g : U_{d,n} \gquot_{\gamma,\vec{c},-} \SL(d+1) \dashrightarrow U_{d,n} \gquot_{\gamma,\vec{c},+} \SL(d+1)$ is the induced birational map, then the divisor class $g_* D$ is $\mathbb{Q}$-Cartier, and $\mathcal{O} (D)$ is relatively ample over $U_{d,n} \gquot_{\gamma,\vec{c},0} \SL(d+1)$.

Because of Corollary \ref{Flips}, it is interesting to ask when the wall-crossing map is regular.  Although we are unable to answer this question at present, we can provide a condition for the map to contract no curves.  If the GIT quotients were normal, this would be sufficient to conclude that the inverse map is regular in precisely this case (see Remark \ref{Normality}).

\begin{proposition}
\label{Regular}
The rational map
$$U_{d,n} \gquot_{\gamma,\vec{c},+} \SL(d+1) \dashrightarrow U_{d,n} \gquot_{\gamma,\vec{c},-} \SL(d+1)$$
contracts no curves if and only if, for every partition $I = I_1 \sqcup \cdots \sqcup I_m$ into at least 3 disjoint sets, we have $\sum_{i=1}^m \sigma ( I_i ) \neq k$.
\end{proposition}

\begin{proof}
By the diagrams above, the map $f^+$ is bijective if and only if the projection map
$$ \overline{M}_{0,n-i+1} ( \mathcal{Z}_{\gamma , \vec{c}_{I} , 1- \epsilon} ) \to \overline{M}_{0,n-i+1} ( \mathcal{Z}_{\gamma , \vec{c}_{I} , \gamma + \epsilon} ) $$
is bijective.  By Proposition \ref{ProjectionIsomorphism}, this is the case if and only if, for every partition $I = I_1 \sqcup \cdots \sqcup I_m$ into at least 3 disjoint sets, we have $\sum_{i=1}^m \sigma ( I_i ) \neq k$.  It follows that the composite rational map $(f^- )^{-1} \circ f^+$ contracts no curves in precisely this case.
\end{proof}

\section{Examples}

In this section we consider examples of the quotients $U_{d,n} \gquot_{ \gamma , \vec{c}} \SL(d+1)$ for specific choices of $( \gamma , \vec{c} ) \in \Delta$.  We will see that many previously constructed compactifications of $M_{0,n}$ arise as such quotients.

\subsection{Hassett's Spaces}

In \cite{Has03}, Hassett constructs the moduli spaces of weighted pointed stable curves $\overline{M}_{0, \vec{c}}$.  A genus 0 marked curve $(X, p_1 , \ldots , p_n )$ is Hassett stable if:
\begin{enumerate}
\item The singularities are at worst nodal;
\item The are no marked points at nodes;
\item The weight at any smooth point is at most 1, and
\item $\omega_X ( \sum_{i=1}^n c_i p_i )$ is ample.
\end{enumerate}
Here we show that each of Hassett's spaces arises as a quotient of $U_{d,n}$.

\begin{theorem}\label{cor:HassettIso}
Let $(\gamma,\vec{c})$ be a linearization such that there are no strictly semistable points and $1 > \gamma > \max\{\frac{1}{2},1-c_1,\ldots,1-c_n\}$.  Then there is an isomorphism
$\overline{M}_{0,\vec{c}} \cong U_{d,n} \gquot_{(\gamma,\vec{c})} \SL(d+1).$
\end{theorem}

\begin{proof}
It is enough to prove the existence of a morphism $U_{d,n} \gquot_{(\gamma,\vec{c})} \SL(d+1) \rightarrow \overline{M}_{0,\vec{c}}$ preserving the interior.  Indeed, both sides are separated, so such a morphism is automatically inverse to the morphism in Proposition \ref{prop:Hassett}.

We claim that the hypotheses imply that the universal family over the semistable locus $(U_{d,n})^{ss}$ is a family of Hassett-stable curves for the weight vector $\vec{c}$.  Indeed,
\begin{itemize}
\item The singularities are at worst nodal, by Corollary \ref{Singularities} and the assumption $\gamma > \frac{1}{2}$;
\item The are no marked points at nodes, by Proposition \ref{SingularWeight} and the fact that $\gamma > 1 - c_i$ for $i=1,\ldots,n$;
\item The weight at any smooth point is at most 1, by Corollary \ref{SmoothWeight}; and
\item The ampleness condition of Hassett-stability is satisfied.
\end{itemize}
The only item here that needs explanation is the last one.  Hassett-stability, in genus zero, requires that the weight of marked points on any component, plus the number of nodes on that component, is strictly greater than 2.  This follows by the same argument as Proposition \ref{NoAutomorphisms}.

Having shown that we have a family of Hassett-stable curves over the semistable locus, the representability of this moduli space implies that we have a morphism $(U_{d,n})^{ss} \rightarrow \overline{M}_{0,\vec{c}}$.  This is clearly $\SL(d+1)$-invariant, so it descends to a morphism from the categorical quotient, which is precisely the GIT quotient: $U_{d,n} \gquot_{(\gamma,\vec{c})} \SL(d+1) \rightarrow \overline{M}_{0,\vec{c}}$.  The interior $M_{0,n}$ is clearly preserved, so this concludes the proof.
\end{proof}

\begin{corollary}
For all $n \geq 3$, there exists $d \geq 1$ such that every Hassett space of $n$-pointed genus zero curves, including $\overline{M}_{0,n}$, is a quotient of $U_{d,n}$.
\end{corollary}

\begin{proof}
Note that there is a chamber structure on the space of weight data
(\cite[\S5]{Has03}). Chambers are separated by hyperplanes
\[
	\{(c_{1}, \cdots, c_{n})|\sum_{i \in I}c_{i} = 1\}
\]
for some $I \subset \{1, 2, \cdots, n\}$.
Therefore we can find $\epsilon > 0$ satisfying the following property:
For any weight datum $\vec{c}$, there is a weight datum $\vec{c'}$ in the same
chamber and $c_{i}' > \epsilon$ for all $i$.
Now we can take $d$ satisfying $\frac{d+1-n}{d-1} > 1 - \epsilon$.
Then this $d$ satisfies
\[
	1 > \frac{d+1-c}{d-1} \ge \frac{d+1-n}{d-1}
	> 1-\epsilon \ge \mathrm{max}\{\frac{1}{2}, 1-c_{1}', \cdots,
	1-c_{n}'\},
\]
for every weight datum $\vec{c'}$.  The result follows immediately from Theorem \ref{cor:HassettIso}, since any Hassett space with weight datum lying on a wall is isomorphic to one with weight datum lying in an adjacent chamber.
\end{proof}

We note the following fact, which was remarked in the introduction:

\begin{corollary}
There exists $L\in\Delta^\circ$ with $U_{n-2,n}\gquot_L\SL(n-1)\cong\overline{M}_{0,n}$.
\end{corollary}

\begin{proof}
The Hassett space $\overline{M}_{0,\vec{c}}$ with $\vec{c}=(\frac{1}{2}+\epsilon,\ldots,\frac{1}{2}+\epsilon)$ is isomorphic to $\overline{M}_{0,n}$ (in fact they have the same universal curves) since no points are allowed to collide.  Thus, it suffices to take a linearization $(\gamma,\vec{c})\in\Delta^\circ$ with $\gamma > \frac{1}{2}$.  Now, $\gamma = \frac{d+1 - (\frac{n}{2} + n\epsilon)}{d-1}$, so $\gamma > \frac{1}{2}$ is equivalent to $d > n - 2n\epsilon - 3$, so indeed for $\epsilon$ small enough we can take $d=n-2$.
\end{proof}

\subsection{Kontsevich-Boggi compactification}

In \cite{Kont92}, Kontsevich described certain topological modiﬁcations of the moduli spaces
$\overline{M}_{g,n}$ which for $g = 0$ were given an algebraic description by Boggi as an alternate compactification of $\overline{M}_{0,n}$ \cite{Bog99}.  This compactification was later independently constructed by Smyth in \cite{Smy09}.  A genus 0 marked curve $(X, p_1 , \ldots , p_n )$ is \emph{Boggi-stable} if:
\begin{enumerate}
\item The singularities are multinodal;
\item There are no marked points at the singular points;
\item There are at least two points on any tail, and
\item There are no unmarked components.
\end{enumerate}
The Boggi space corresponds to the extremal assignment in which all components without marked points are assigned.  We will see that the Boggi space also arises as a quotient of $U_{d,n}$, in the case $d=n$, $c_i = 1 - \epsilon$ $\forall i$.  Note that in this case $\gamma = \frac{1+d \epsilon}{d-1}$.

\begin{proposition}
The GIT quotient $U_{d,n} \gquot_{\frac{1+d \epsilon}{d-1}, \vec{1 - \epsilon}}\SL(d+1)$ is isomorphic to the Boggi space $\overline{M}_{0,n}^{Bog}$.
\end{proposition}

\begin{proof}
Let $(X, p_1 , \ldots , p_n ) \in \overline{M}_{0,n}$ be a Deligne-Mumford stable curve.  It suffices to show that a component of $X$ is $\mathcal{Z}_{\frac{1+d \epsilon}{d-1}, \vec{1 - \epsilon}}$-assigned if and only if it is unmarked.  Let $Y \subset X$ be a tail containing $k$ marked points.  Then
$$ \sigma (Y) = \lceil \frac{k(1 - \epsilon ) - 1}{1 - \gamma} \rceil = k .$$
Hence, for any component $E \subset X$, $E$ is assigned if and only if the total number of points on the connected components of $\overline{X \backslash E}$ is equal to $d=n$.  In other words, $E$ is assigned if and only if it is unmarked.
\end{proof}

\subsection{Variation of GIT}

In addition to previously constructed moduli spaces, our GIT approach also recovers known maps between these moduli spaces.  As an example we consider the case where $n=d=9$ and the weights are symmetric -- that is, $c_i = c_j$ $\forall i,j$.  By the results above, we see that $U_{9,9} \gquot_{\gamma , \vec{c}}\SL(10)$ is isomorphic to a Hassett space for all $\gamma > \frac{1}{2}$, and isomorphic to the Boggi space for $\frac{1}{9} < \gamma < \frac{2}{7}$.  In the range $\frac{2}{7} < \gamma < \frac{1}{2}$, the space $\overline{M}_{0,9}^{trip} = U_{9,9} \gquot_{\gamma , \vec{c}}\SL(10)$ is isomorphic to $\overline{M}_{0,9}$, but the corresponding moduli functor is different.  Specifically, a curve consisting of three components meeting in a triple point, each containing three marked points, is GIT-stable, while the corresponding Deligne-Mumford stable curve obtained by replacing the triple point with a rational triborough is not GIT-stable.  We note furthermore that since all of the moduli spaces just described are normal, the corresponding wall-crossing maps are all regular by Proposition \ref{Regular}.  As we increase $\gamma$ from $\frac{1}{9}$ to 1, we therefore obtain the following picture:

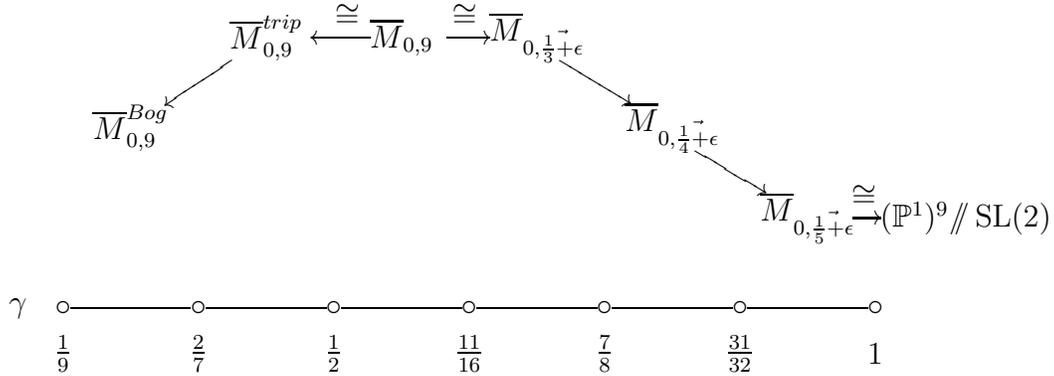
\[
\begin{xy} 0;<6mm,0mm>:
	0*{\gamma};(1,0)*{\circ}; (4,0)*{\circ}**@{-};
	(7,0)*{\circ}**@{-}; (10,0)*{\circ}**@{-};
	(13,0)*{\circ}**@{-}; (16,0)*{\circ}**@{-};
	(19,0)*{\circ}**@{-};
	(1,-1)*{\frac{1}{9}};(4,-1)*{\frac{2}{7}};
	(7,-1)*{\frac{1}{2}};(10,-1)*{\frac{11}{16}};
	(13,-1)*{\frac{7}{8}};(16,-1)*{\frac{31}{32}};
	(19,-1)*{1};
	(2.5, 4)*{\overline{M}_{0,9}^{Bog}};
	(5.5, 6)*{\overline{M}_{0,9}^{trip}}**@{-}?<*@{<};(6.5,6);
	(8.5,6)*{\overline{M}_{0,9}}**@{-}?<*@{<}?(0.4)*!/_3mm/{\cong};
	(9.5,6);
	(11.5,6)*{\overline{M}_{0,\vec{\frac{1}{3}+\epsilon}}}**@{-}
	?(0.2)*!/_3mm/{\cong} ?>*@{>};(12,5.5);
	(14.5,4)*{\overline{M}_{0,\vec{\frac{1}{4}+\epsilon}}}**@{-}
	?>*@{>};(15,3.5);
	(17.5,2)*{\overline{M}_{0,\vec{\frac{1}{5}+\epsilon}}}**@{-}
	?>*@{>};(18.5, 2);
	(21,2)*{(\PP^{1})^{9}\gquot \SL(2)}**@{-}
	?(0.1)*!/_3mm/{\cong}?>*@{>};
\end{xy}
\]

\subsection{An Example of a Flip}

While the previous example includes several previously constructed spaces, it does not include any flips.  To see an example of a flip, we consider the case where $d=5, n=19$, and the weights are symmetric.  Let $I_k$ denote any set of $k$ marked points.  When $\gamma = \frac{4}{9} + \epsilon$, we see that
\begin{displaymath}
\sigma ( I_k ) = \left\{ \begin{array}{ll}
0 & \textrm{if $k \leq 4$}\\
1 & \textrm{if $5 \leq k \leq 7$}\\
2 & \textrm{if $8 \leq k \leq 9$}\\
3 & \textrm{if $10 \leq k \leq 11$}\\
4 & \textrm{if $12 \leq k \leq 14$}\\
5 & \textrm{if $15 \leq k$}
\end{array} \right.
\end{displaymath}
On the other hand, when $\gamma = \frac{4}{9} - \epsilon$, then each of these remains the same, except for $ \sigma ( I_7 )$ which becomes 2, and $\sigma ( I_{12} )$, which becomes 3.  Now, consider the diagram
$$\xymatrix{
U_{5,19} \gquot_{\frac{4}{9} + \epsilon , \vec{\frac{2}{9} - \epsilon}} \SL(6) \ar@{<.>}[rr] \ar[rd]^{f^+} & & U_{5,19} \gquot_{\frac{4}{9} - \epsilon , \vec{\frac{2}{9} + \epsilon}} \SL(6) \ar[ld]_{f^-} \\
 & U_{5,19} \gquot_{\frac{4}{9} , \vec{\frac{2}{9}}} \SL(6) & . }$$
By Corollary \ref{DivisorialContraction}, neither $f^+$ nor $f^-$ contracts a divisor.  On the other hand, the map $f^+$ contracts the F-curve class $(10,7,1,1)$, whereas the map $f^-$ contracts the F-curve class $(12,5,1,1)$, so neither $f^+$ nor $f^-$ is trivial.  (The numerical class of an F-curve is determined by the number of marked points on each leg, whence the preceding notation.) It follows from Corollary \ref{Flips} that the diagram is a flip.

Finally we note that the moduli space $U_{5,19} \gquot_{\frac{4}{9}, \vec{\frac{2}{9}}}\SL(6)$ is not isomorphic to a modular compactification as in \cite{Smy09} (this does not contradict Proposition \ref{prop:GITass} because the linearization lies on a GIT wall, hence there are strictly semistable points).  In this sense it is truly a ``new'' compactification of $M_{0,19}$.  To see this, consider the Deligne-Mumford stable curve $X$ which is a chain of 4 rational curves, each component containing 10, 2, 2, and 5 marked points, respectively.  The image of $X$ in the GIT quotient has three components.  These components have 10, 0, and 5 marked points on their interiors, and there are 2 marked points at each of the nodes -- the two interior components of $X$ are contracted.  On the other hand, the original curve is a specialization of a Deligne-Mumford stable curve $Y$ consisting of 3 components, containing 10, 4, and 5 marked points, respectively.  Hence, if this space were modular, then by \cite{Smy09} the interior component of $Y$ would have to be contracted as well.  But we see that this is not the case.

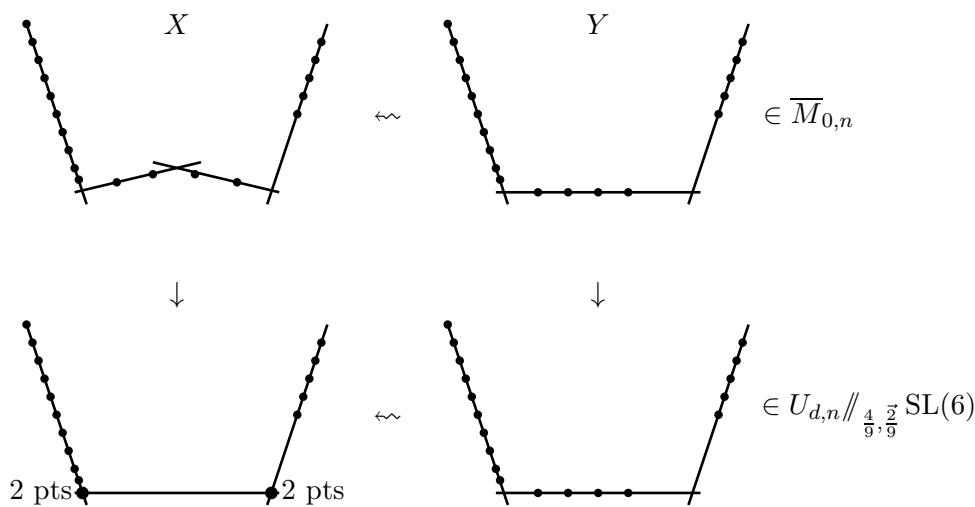
\begin{center}
\small
\begin{tikzpicture}[scale=0.8]
	\draw[line width=1pt] (0, 10) -- (1, 7);
	\draw[line width=1pt] (0.8, 7.2) -- (2.9, 7.7);
	\draw[line width=1pt] (2.1, 7.7) -- (4.2, 7.2);
	\draw[line width=1pt] (4, 7) -- (5, 10);
	\fill (0, 10) circle (2pt);
	\fill (0.1, 9.7) circle (2pt);
	\fill (0.2, 9.4) circle (2pt);	
	\fill (0.3, 9.1) circle (2pt);	
	\fill (0.4, 8.8) circle (2pt);	
	\fill (0.5, 8.5) circle (2pt);
	\fill (0.6, 8.2) circle (2pt);	
	\fill (0.7, 7.9) circle (2pt);	
	\fill (0.8, 7.6) circle (2pt);	
	\fill (0.87, 7.41) circle (2pt);	
	\fill (1.5, 7.366) circle (2pt);
	\fill (2.1, 7.5) circle (2pt);
	\fill (3.5, 7.366) circle (2pt);
	\fill (2.8, 7.5) circle (2pt);
	\fill (4.9, 9.7) circle (2pt);
	\fill (4.8, 9.4) circle (2pt);	
	\fill (4.7, 9.1) circle (2pt);	
	\fill (4.6, 8.8) circle (2pt);	
	\fill (4.5, 8.5) circle (2pt);
	
	\draw[line width=1pt] (7, 10) -- (8, 7);
	\draw[line width=1pt] (7.8, 7.2) -- (11.2, 7.2);
	\draw[line width=1pt] (11, 7) -- (12, 10);

	\fill (7, 10) circle (2pt);
	\fill (7.1, 9.7) circle (2pt);
	\fill (7.2, 9.4) circle (2pt);	
	\fill (7.3, 9.1) circle (2pt);	
	\fill (7.4, 8.8) circle (2pt);	
	\fill (7.5, 8.5) circle (2pt);
	\fill (7.6, 8.2) circle (2pt);	
	\fill (7.7, 7.9) circle (2pt);	
	\fill (7.8, 7.6) circle (2pt);	
	\fill (7.87, 7.41) circle (2pt);	
	\fill (8.5, 7.2) circle (2pt);
	\fill (9, 7.2) circle (2pt);
	\fill (9.5, 7.2) circle (2pt);
	\fill (10, 7.2) circle (2pt);
	\fill (11.9, 9.7) circle (2pt);
	\fill (11.8, 9.4) circle (2pt);	
	\fill (11.7, 9.1) circle (2pt);	
	\fill (11.6, 8.8) circle (2pt);	
	\fill (11.5, 8.5) circle (2pt);

	\draw[line width=1pt] (7, 5) -- (8, 2);
	\draw[line width=1pt] (7.8, 2.2) -- (11.2, 2.2);
	\draw[line width=1pt] (11, 2) -- (12, 5);

	\fill (7, 5) circle (2pt);
	\fill (7.1, 4.7) circle (2pt);
	\fill (7.2, 4.4) circle (2pt);	
	\fill (7.3, 4.1) circle (2pt);	
	\fill (7.4, 3.8) circle (2pt);	
	\fill (7.5, 3.5) circle (2pt);
	\fill (7.6, 3.2) circle (2pt);	
	\fill (7.7, 2.9) circle (2pt);	
	\fill (7.8, 2.6) circle (2pt);	
	\fill (7.87, 2.41) circle (2pt);	
	\fill (8.5, 2.2) circle (2pt);
	\fill (9, 2.2) circle (2pt);
	\fill (9.5, 2.2) circle (2pt);
	\fill (10, 2.2) circle (2pt);
	\fill (11.9, 4.7) circle (2pt);
	\fill (11.8, 4.4) circle (2pt);	
	\fill (11.7, 4.1) circle (2pt);	
	\fill (11.6, 3.8) circle (2pt);	
	\fill (11.5, 3.5) circle (2pt);

	\draw[line width=1pt] (0, 5) -- (1, 2);
	\draw[line width=1pt] (0.8, 2.2) -- (4.2, 2.2);
	\draw[line width=1pt] (4, 2) -- (5, 5);

	\fill (0, 5) circle (2pt);
	\fill (0.1, 4.7) circle (2pt);
	\fill (0.2, 4.4) circle (2pt);	
	\fill (0.3, 4.1) circle (2pt);	
	\fill (0.4, 3.8) circle (2pt);	
	\fill (0.5, 3.5) circle (2pt);
	\fill (0.6, 3.2) circle (2pt);	
	\fill (0.7, 2.9) circle (2pt);	
	\fill (0.8, 2.6) circle (2pt);	
	\fill (0.87, 2.41) circle (2pt);
	\fill (4.9, 4.7) circle (2pt);
	\fill (4.8, 4.4) circle (2pt);	
	\fill (4.7, 4.1) circle (2pt);	
	\fill (4.6, 3.8) circle (2pt);	
	\fill (4.5, 3.5) circle (2pt);
	
	\fill (0.93, 2.2) circle (3pt) node[left] {2 pts};	
	\fill (4.07, 2.2) circle (3pt) node[right] {2 pts};	
	
	\node (x) at (2.5, 10) {$X$};
	\node (y) at (9.5, 10) {$Y$};
	\node (m) at (13, 8.5) {$\in \overline{M}_{0,n}$};
	
	\node (a) at (6, 8.5)
	{\rotatebox[origin=c]{180}{$\rightsquigarrow$}};
	\node (b) at (6, 3.5)
	{\rotatebox[origin=c]{180}{$\rightsquigarrow$}};
	\node (c) at (2.5, 5.5) {$\downarrow$};
	\node (d) at (9.5, 5.5) {$\downarrow$};
	\node (u) at (14, 3.5) {$\in U_{d,n}\gquot_{\frac{4}{9},
	\vec{\frac{2}{9}}}\SL(6)$};
\end{tikzpicture}
\normalsize
\end{center}

\subsection{Modular compactifications not from GIT}\label{section:ModNotGIT}

In the above subsection we saw an example of a GIT compactification of $M_{0,n}$ which is not modular in the sense of \cite{Smy09}.  On the other hand, there are also examples of modular compactifications which do not arise from our GIT construction.  For instance, consider a partition $[n]=I\sqcup J\sqcup K$ into three nonempty subsets.  It is easy to see that assigning a tail if and only if the marked points on it are indexed entirely by $I$ or entirely by $J$ yields an extremal assignment.  Suppose this assignment is given by a geometric quotient of $U_{d,n}$.  If a tail has only two marked points, $p_{i_1},p_{i_2}$, both indexed by $I$, then by Proposition \ref{DegreeOfTails} we have $\sigma(\{i_1,i_2\})=0$ and so $c_{i_1}+c_{i_2} < 1$.  Similarly, considering a tail with two points $p_{j_1},p_{j_2}$ both indexed by $J$ forces $c_{i_1}+c_{i_2} < 1$.  Without loss of generality write $c_{i_1} \le c_{i_2}$ and $c_{j_1} \le c_{j_2}$.  Then $c_{i_1}+c_{j_1} < 1$, so $\sigma(\{i_1,j_1\})=0$, and hence a tail with only $p_{i_1}$ and $p_{j_1}$ would be contracted, contradicting the definition of the extremal assignment.

\bibliography{ref}

\end{document}